\documentclass[12pt]{amsart}

\usepackage{amsmath, amssymb, euscript,  enumerate}
\usepackage{color}
\usepackage{epsfig,graphicx}
\usepackage{ulem}
\usepackage{cite}
\usepackage{pxfonts}
\usepackage{color,wrapfig}

\usepackage{hyperref}
\usepackage{url}

  \newtheorem{theorem}{Theorem}
  \newtheorem{lemma}{Lemma}
  \newtheorem{cor}{Corollary}
  \newtheorem{remark}{Remark}
  
  \newtheorem{example}{Example}
  \newtheorem{definition}{Definition}

  \numberwithin{equation}{section}

\begin{document}

\title[Generalizations of the Choe-Hoppe helicoid]{Generalizations of the 
Choe-Hoppe helicoid and Clifford cones in Euclidean space} 

\bigskip 

 \author{Eunjoo Lee}
\address{Korea Institute for Advanced Study School of Mathematics, Hoegiro 85, Dongdaemun-gu, Seoul 130-722, Korea}
\email{eunjoolee@kias.re.kr}

\author{Hojoo Lee}
\address{Korea Institute for Advanced Study School of Mathematics, Hoegiro 85, Dongdaemun-gu, Seoul 130-722, Korea}
\email{momentmaplee@gmail.com, autumn@kias.re.kr}
 

\begin{abstract}
By sweeping out $L$ indpendent Clifford cones in ${\mathbb{R}}^{2N+2}$ via the multi-screw motion, 
we construct minimal submanifolds in  ${\mathbb{R}}^{L(2N+2)+1}$. Also, we sweep out the $L$-rays 
Clifford cone (introduced in Section \ref{section multi rays}) in ${\mathbb{R}}^{L(2N+2)}$ to construct minimal submanifolds  in ${\mathbb{R}}^{L(2N+2)+1}$. Our minimal submanifolds unify various interesting examples: Choe-Hoppe's helicoid of codimension one, cone over Lawson's ruled minimal surfaces in ${\mathbb{S}}^{3}$, Barbosa-Dajczer-Jorge's ruled submanifolds, and Harvey-Lawson's volume-minimizing twisted normal cone over the Clifford torus $ \frac{1}{\sqrt{2}}  {\mathbb{S}}^{N} \times   \frac{1}{\sqrt{2}}  {\mathbb{S}}^{N}$.   
\end{abstract}


\maketitle

  \begin{center}
  {\small{\textit{Dedicated to Professor Jaigyoung Choe in honor of his $61$st birthday}}}
  \end{center}

\bigskip
\bigskip

  \tableofcontents
 

 \section{Helicoids and minimal cones}

In 1867, Riemann \cite{Riemann}  discovered a family of complete, embedded, singly periodic 
minimal surfaces in Euclidean space  ${\mathbb{R}}^{3}$ foliated by circles and lines. He established 
that his staircases, planes, catenoids, and helicoids  are the only minimal surfaces fibered by circles or 
lines in parallel planes. 

Catenoids can be generalized to higher dimensions and have been characterized in various ways. 
Higher dimensional catenoids in ${\mathbb{R}}^{K \geq 4}$ are the minimal hypersurfaces spanned by 
a family of coaxial $(K-2)$-dimensional round spheres of varying radii. 
In 1991, Jagy \cite{Jagy1991}  adopted Schoen's argument \cite{Schoen1983} to show that if a minimal 
hypersurface ${\Sigma}^{K-1}$ in  ${\mathbb{R}}^{K \geq 4}$ is foliated by $(K-2)$-dimensional round 
spheres of varying radii in parallel $(K-1)$-dimensional hyperplanes, then the submanifold ${\Sigma}^{K-1}$ should be
rotationally symmetric. See also Shiffman's Theorem \cite[Theorem 1]{Shiffman1956}.

While there were several interesting results on higher dimensional catenoids, significant results on helicoids have been mostly in  ${\mathbb{R}}^{3}$.
For instance, Colding-Minicozzi's deep description illustrates that 
embedded minimal disks in a ball in  ${\mathbb{R}}^{3}$ are modeled on planes or helicoids  \cite{CM01, CM02, CM03, CM04}.
Meeks and Rosenberg \cite{MR2005} proved that helicoids and planes are the only complete, properly embedded, simply connected minimal surfaces.  
Bernstein and Breiner \cite{BC2001} used the Colding-Minicozzi theory to show that a  complete, properly embedded, minimal surface with finite genus and one end must be asymptotic to a helicoid. 

 Helicoids in  ${\mathbb{R}}^{3}$ can be characterized in various ways. 
 For instance, we can obtain helicoids by taking conjugate surfaces of catenoids. However, the notion of conjugation (for instance, see \cite{Eschenburg2006, KGB2001}) is not known for minimal hypersurfaces in  ${\mathbb{R}}^{K \geq 4}$.
 Also, Catalan's Theorem shows that a ruled minimal surface in ${\mathbb{R}}^{3}$ should be a helicoid 
\[
 {\mathcal{H}}_{\left({\lambda}_{0},  {\lambda}_{1}\right)} =  \left\{ \; 
 \begin{bmatrix} \; \mathbf{r} \; e^{ \, i  \left( {\lambda}_{1}  \Theta \right) \, }  \\  {\lambda}_{0}  \Theta  \end{bmatrix} \in 
{\mathbb{C}} \times \mathbb{R} \; \vert \;  \Theta, \mathbf{r}  \in \mathbb{R}
\right\},  \;\; {\lambda}_{0},  {\lambda}_{1} \in \mathbb{R} \text{: constants}.
\]
The helicoid ${\mathcal{H}}_{\left({\lambda}_{0},  {\lambda}_{1}\right)}$ in  ${\mathbb{R}}^{3}$ is invariant under the screw motion 
\[
\begin{bmatrix} \;   x+iy\; \\  z   \end{bmatrix} \in 
{\mathbb{C}} \times \mathbb{R}  \; \mapsto \;
\begin{bmatrix} \;   e^{ \, i  \left( {\lambda}_{1}  t \right) \, } (x+iy) \;  \\  z+ {\lambda}_{0} t   \end{bmatrix} \in  {\mathbb{C}} \times \mathbb{R}. 
\]
This geometric observation gives an insight on generalizing classical helicoids into higher dimensions as in  \cite{CH2013}. 

Choe and Hoppe \cite[Theorem 2]{CH2013} gave an explicit construction of a minimal hypersurface in   ${\mathbb{R}}^{2N+1}$  foliated by Clifford hypercones  in ${\mathbb{R}}^{2N}$. 
The Choe-Hoppe helicoid in  ${\mathbb{R}}^{2N+1}$ is the hypersurface 
\[
 z= f \left( x_{1}, y_{1}, \cdots, x_{N}, y_{N} \right) = \mathbf{arg} \, \left(  \sqrt{  \; (x_{1}+i y_{1})^{2}  + \cdots + 
  (x_{N}+i y_{N})^{2}  \;  } \;  \right),
\]
up to homotheties. Recently, Del Pino, Musso, and Pacard \cite{PMP2012} produced new solutions of the Allen-Cahn equation whose zero set is the Choe-Hoppe helicoid. 
See also Wei-Yang's traveling wave solutions with vortex helix structure for Schr\"{o}dinger map equation 
\cite{WY2013}.

Minimal cones  play an important role in solving higher dimensional Bernstein problems (for instance, see Fleming's argument \cite{Fleming1962}  and Simons' Theorem\cite[Theorem 6.2.2]{Simons1968}) and understanding 
the nature of singularities of minimal  varieties.  Smale \cite{Smale1989} used  disjoint stable minimal hypercones in ${\mathbb{R}}^{K \geq 8}$ to construct many stable embedded minimal hypersurfaces with boundary, in ${\mathbb{R}}^{K \geq 8}$, with an arbitrary number of isolated singularities and 
prescribed rate of decay to their tangent cones at the singularities. Moreover,  minimal cones in the unit ball become important examples of free boundary minimal varieties. See papers \cite{FR2011, FR2012} by Fraser and Schoen, and the survey \cite[Example 2.10]{Schoen2013} by Schoen.

Our main goal is to generalize Choe-Hoppe's minimal variety. By sweeping out 
Clifford cones or  multi-rays Clifford cones (Definition \ref{GCC}), we explicitly construct 
 generalized helicoids in 
odd dimensional Euclidean spaces  (Theorem \ref{gHEL1} and Theorem \ref{gHEL2}) and new 
minimal cones in even dimensional Euclidean spaces  (Corollary \ref{minimal multi-lays clifford cones}, 
Remark \ref{CONE by THM 2}, and Example \ref{E3}). We also extend Takahashi's classical criterion to higher 
codimension (Theorem \ref{ET}). We shall show that our minimal submanifolds naturally unify various  
minimal submanifolds in Euclidean space and 
unit sphere. See four examples illustrated in Section \ref{MAINresults}.

 \section{Preliminaries}     \label{Pre}

\subsection{Multi-screw motions in Euclidean space}

For a given angle $\theta$ and a complex vector
\[
X + i  Y =   \begin{bmatrix} x_{1} + i y_{1} \\ \vdots \\  x_{N+1} + i y_{N+1} \end{bmatrix}  \in  {\mathbb{C}}^{N+1} = {\mathbb{R}}^{N+1} + i{\mathbb{R}}^{N+1},
\]
we adopt the notation  
\[
e^{i( \lambda_{i} \theta )} \left( X+ i  Y  \right) = \begin{bmatrix} e^{i( \lambda_{i} \theta )}  \left( x_{1} + i y_{1} \right) \\ \vdots \\  e^{i( \lambda_{i} \theta )}  \left( x_{N+1} + i y_{N+1} \right) \end{bmatrix}.
\] 
We also use the complex structure $\mathbf{J}$ as the $\frac{\pi}{2}$-rotation. More explicitly,  
\[
\mathbf{J} \left( X+ i  Y  \right) = i \left( X+ i  Y  \right)  = \begin{bmatrix}  - y_{1} + i x_{1}   \\ \vdots \\    -y_{N+1} + i x_{N+1}  \end{bmatrix}.
\] 
We then introduce the multi-screw motion in ${\mathbb{R}}^{L(2N+2)+1}$.

\begin{definition}[\textbf{Multi-screw motion in Euclidean space ${\mathbb{R}}^{L(2N+2)+1}$}]  \label{MSM}
Let $L \geq 1$ and $N \geq 0$ be integers. Given an $(L+1)$-tuple of real numbers 
\[
{\Lambda}= \left( {\lambda}_{0}, {\lambda}_{1}, \cdots, {\lambda}_{L} \right),
\]
we introduce the multi-screw motion ${\mathcal{S}}_{{\Lambda}}$ in ${\mathbb{R}}^{L(2N+2)+1} = {\mathbb{C}}^{(N+1)L} \times {\mathbb{R}}$ with the pitch vector ${\Lambda}$. The mapping ${\mathcal{S}}_{{\Lambda}}$ is defined by 
\begin{equation} \label{HEL}
\begin{bmatrix}
{X_{1}} + i  {Y_{1}}  \\ 
\vdots \\
X_{L}   + i  {Y_{L}} \\ 
z
\end{bmatrix}
\mapsto
\begin{bmatrix}
 e^{i( \lambda_{1} \theta )} \left( {X_{1}} + i  {Y_{1}}   \right) \\ 
\vdots \\
 e^{i( \lambda_{L} \theta )} \left( {X_{L}} + i  {Y_{L}}   \right) \\ 
z +  {\lambda_{0}} \theta
\end{bmatrix}.
\end{equation}

\end{definition}

\subsection{Identities on higher dimensional Clifford tori}

\begin{definition}[\textbf{Higher dimensional Clifford tori in sphere ${\mathbb{S}}^{2N+1}$}] 
A $2N$-dimensional Clifford torus $\mathcal{C} = \frac{1}{\sqrt{2}}  {\mathbb{S}}^{N} \times   \frac{1}{\sqrt{2}}  {\mathbb{S}}^{N} $ denotes the minimal hypersurface in ${\mathbb{S}}^{2N+1} \subset {\mathbb{R}}^{2N+2}$ defined by
\[
 \mathcal{C} =\left\{ \frac{1}{\sqrt{2}} \begin{bmatrix} X \\ Y
\end{bmatrix}  \in {\mathbb{S}}^{2N+1} \subset  {\mathbb{R}}^{2(N+1)} \; \vert \;  {\Vert X \Vert}_{  {\mathbb{R}}^{N+1} } =1, \; {\Vert Y \Vert}_{  {\mathbb{R}}^{N+1} } =1 \right\} .
\]
 \end{definition}

Throughout this article, the symbol $\;\cdot\;$ means the standard dot product in Euclidean spaces.

\begin{lemma}[\textbf{Identities on  Clifford tori $\frac{1}{\sqrt{2}}  {\mathbb{S}}^{N} \times   \frac{1}{\sqrt{2}}  {\mathbb{S}}^{N}$}]
\label{magic}
Let $\mathcal{C}$ denote a local parameterization of the Clifford torus $\frac{1}{\sqrt{2}}  {\mathbb{S}}^{N} \times   \frac{1}{\sqrt{2}}  {\mathbb{S}}^{N} $  in ${\mathbb{S}}^{2N+1} \subset {\mathbb{R}}^{2N+2}$
\[
   \mathcal{C} \left( u_{1}, \dots, u_{2N} \right) = \frac{1}{\sqrt{2}} \begin{bmatrix} X \left( u_{1}, \dots, u_{N} \right) \\ Y
\left( u_{N+1}, \dots, u_{2N} \right) \end{bmatrix} \in {\mathbb{R}}^{2N+2}, 
\]
where $X \left( u_{1}, \dots, u_{N} \right)$ and $Y \left( u_{N+1}, \dots, u_{2N} \right)$ are an ${\mathbb{R}}^{N+1}$-valued local chart of two independent unit spheres of dimension $N$. Its unit normal vector reads 
\[
   \mathcal{D} \left( u_{1}, \dots, u_{2N} \right) = \frac{1}{\sqrt{2}} \begin{bmatrix} X \left( u_{1}, \dots, u_{N} \right) \\  - Y
\left( u_{N+1}, \dots, u_{2N} \right) \end{bmatrix}.
\] 
Let $\left(  {\mathbf{g}}^{ij}  \right)_{ 1 \leq i,j \leq 2N }$ denote the inverse of matrix  $ \left(  {\mathbf{g}}_{ij}  \right)_{1 \leq i,j \leq 2N}$ of the first fundamental form induced by the immersion $\mathcal{C}$ in coordinates $u_{1}, \dots, u_{2N}$. We also use the abbreviation ${\mathbf{g}} = \mathrm{det} \, \left(  {\mathbf{g}}_{ij}  \right)$ and introduce the $\mathbb{R}$-valued function
\[
      w_{j} =  \frac{\partial  \mathcal{C}}{\partial u_j} \cdot \mathbf{J} \mathcal{C}, \quad j \in \left\{1, \cdots, 2N \right\}.
\]
Then, we have the following identities:
 
\begin{enumerate}
\item[\textbf{(a)}] 
\[
\mathbf{J} \mathcal{C} =  \left(\mathcal{D} \cdot \mathbf{J}  \mathcal{C} \right) \; \mathcal{D} + 
\sum_{1 \leq i,j \leq 2N} g^{ij} w_{j} \frac{\partial \mathcal{C}}{\partial u_{i}}, 
\] 
and
\[
- \mathbf{J} \mathcal{D} =  \left(\mathcal{D} \cdot \mathbf{J}  \mathcal{C} \right) \; \mathcal{C} + 
\sum_{1 \leq i,j \leq 2N} g^{ij} w_{j} \frac{\partial \mathcal{D}}{\partial u_{i}}. 
\]
\item[\textbf{(b)}]
\[
1 - \left(  {  \mathcal{D} \cdot \mathbf{J}  \mathcal{C}     }   \right)^{2} = \sum_{1 \leq i,j \leq 2N}  {\mathbf{g}}^{ij} w_{i} w_{j}. 
\]
\item[\textbf{(c)}]
\[
 \sum_{1 \leq i,j \leq 2N}  \frac{\partial}{\partial u_{i}} \left(   \sqrt{  \mathbf{g} \, } \;   {\mathbf{g}}^{ij} w_{j} \right) = 0.
\]
\item[\textbf{(d)}]
\[
\sum_{1 \leq i,j \leq 2N}   {\mathbf{g}}^{ij}  w_{j}  \frac{\partial }{\partial u_{i}}  \left(  {  \mathcal{D} \cdot \mathbf{J}  \mathcal{C}     }   \right) = 0.
\]
\item[\textbf{(e)}]
\[
\sum_{1 \leq i,j \leq 2N}   {\mathbf{g}}^{ij}     \frac{\partial }{\partial u_{i}}  \left(  {  \mathcal{D} \cdot \mathbf{J}  \mathcal{C}     }   \right) \frac{\partial \mathcal{C}}{\partial u_{j}} =  - 2 \left(\, \mathbf{J} \mathcal{D} +  \left(\mathcal{D} \cdot \mathbf{J}  \mathcal{C} \right) \, \mathcal{C} \, \right).
\]
\end{enumerate}
\end{lemma}

\begin{proof} From the definition of the immersion $\mathcal{C}$, we observe that the symmetric matrix $\left(  {\mathbf{g}}_{ij}  \right)_{1 \leq i,j \leq 2N}$ becomes a block matrix  such that ${\mathbf{g}}_{ij}={\mathbf{g}}_{ji}=0$  holds whenever $i \leq N$ and $j \geq N+1$.

\textbf{(a)} Since the second identity is equivalent to the first one, we only check the first one. Fix the coordinates $\left(u_{1}, \cdots, u_{2N}\right)$.  Since vectors $X$, $\frac{\partial X}{\partial{u_{1}}}$,  $\cdots$, $\frac{\partial X}{\partial{u_{N}}}$ in 
${\mathbb{R}}^{N+1}$ are linearly independent, we have the linear combination
\[
 - Y = - \left( Y \cdot X \right) X + \sum_{i=1}^{N} \tau_{i} \frac{\partial X}{\partial{u_{i}}},
\]
and similarly, 
\[
 \; \, X = \; \,  \left( X \cdot Y \right) Y + \sum_{i=N+1}^{2N} \tau_{i} \frac{\partial Y}{\partial{u_{i}}},
\]
for some coefficients ${\tau}_{1}$, $\cdots$, ${\tau}_{2N}$.
Combining these two, we have the linear combination  
\[
\mathbf{J} \mathcal{C} =  \frac{1}{\sqrt{2}} \begin{bmatrix} -Y  \\ X  \end{bmatrix} = \left(\mathcal{D} \cdot \mathbf{J}  \mathcal{C} \right) \; \mathcal{D} + 
\sum_{i=1}^{2N} \tau_{i} \frac{\partial \mathcal{C}}{\partial u_{i}}. 
\]
Taking the dot product with the vector $\frac{\partial \mathcal{C}}{\partial u_{j}}$ in both sides yields
\[
 w_{j} =  \frac{\partial  \mathcal{C}}{\partial u_j} \cdot \mathbf{J} \mathcal{C} = \sum_{i=1}^{2N}  {\mathbf{g}}_{ji}  \tau_{i}, \quad j \in \left\{1, \cdots, 2N \right\},
\] 
which implies that 
\[
 {\tau}_{i} = \sum_{j=1}^{2N}  {\mathbf{g}}^{ij} w_{j}, \quad i \in \left\{1, \cdots, 2N \right\}.
\]
\textbf{(b)} It follows from $\textbf{(a)}$ and $ \mathcal{C} \cdot  \mathcal{C}=1$ that  
\begin{eqnarray*}
1 - \left(  {  \mathcal{D} \cdot \mathbf{J}  \mathcal{C}     }   \right)^{2}  &=&
 \mathbf{J}  \mathcal{C}  \cdot \left[ \, \mathbf{J} \mathcal{C} -  \left(\mathcal{D} \cdot \mathbf{J}  \mathcal{C} \right)  \mathcal{D}  \, \right]  \\
 &=& \mathbf{J}  \mathcal{C}  \cdot \left[  \, \sum_{1 \leq i,j \leq 2N} g^{ij} w_{j}  \frac{\partial \mathcal{C}}{\partial u_{i}}  \, \right] \\
 &=&  \sum_{1 \leq i,j \leq 2N}  {\mathbf{g}}^{ij} w_{j} w_{i}. 
\end{eqnarray*}
\textbf{(c)} We begin with the decomposition 
\begin{eqnarray*}
&&  \sum_{1 \leq i,j \leq 2N}  \frac{\partial}{\partial u_{i}} \left(   \sqrt{  \mathbf{g} \, } \;   {\mathbf{g}}^{ij} w_{j} \right) \\
&=& \sum_{1 \leq i,j \leq 2N}  \frac{\partial}{\partial u_{i}} \left[   \sqrt{  \mathbf{g} \, } \;   {\mathbf{g}}^{ij}  \; \left( 
 \frac{\partial  \mathcal{C}}{\partial u_j} \cdot \mathbf{J} \mathcal{C} \right) \;  \right] \\
& = &
\sum_{1 \leq i,j \leq 2N}  \frac{\partial}{\partial u_{i}} \left[ \;  \left( \sqrt{  \mathbf{g} \, } \;   {\mathbf{g}}^{ij} \; \frac{\partial   \mathcal{C}}{\partial u_j}  \right) \cdot \mathbf{J} \mathcal{C}  \;  \right] \\
& =&  \; \sum_{1 \leq i,j \leq 2N} \; \left[ \frac{\partial}{\partial u_{i}} \left(  \sqrt{  \mathbf{g} \, } \;   {\mathbf{g}}^{ij} \;    \frac{\partial   \mathcal{C}}{\partial u_j}    \; \right)  \; \right] \cdot  \mathbf{J} \mathcal{C}   \; + 
\sum_{1 \leq i,j \leq 2N} \;  \left(   \sqrt{  \mathbf{g} \, } \;   {\mathbf{g}}^{ij}   \frac{\partial   \mathcal{C}}{\partial u_j}  \right) \; \cdot \frac{\partial   }{\partial u_i} \left(  \mathbf{J} \mathcal{C} \right).
\end{eqnarray*}
Since the Clifford torus $\frac{1}{\sqrt{2}}  {\mathbb{S}}^{N} \times   \frac{1}{\sqrt{2}}  {\mathbb{S}}^{N} $  in minimal ${\mathbb{S}}^{2N+1} \subset {\mathbb{R}}^{2N+2}$, its mean curvature vector vanishes:
\[
   {\triangle}_{g_{\mathbf{C}}}  \mathbf{C} + 2N \mathbf{C} \equiv \mathbf{0},
\]
which implies that the first sum vanishes:
\[
 \left[ \; \sum_{1 \leq i,j \leq 2N}  \frac{\partial}{\partial u_{i}} \left(   \sqrt{  \mathbf{g} \, } \;   {\mathbf{g}}^{ij} \;   \frac{\partial   \mathcal{C}}{\partial u_j}    \right)  \; \right] \cdot  \mathbf{J} \mathcal{C} = (-2N \mathcal{C} ) \cdot \mathbf{J} \mathcal{C} =0.
\]
By using the fact that the matrix $\left(  {\mathbf{g}}_{ij}  \right)_{1 \leq i,j \leq 2N}$ is symmetric and by
noticing  that 
\[
 \frac{\partial   \mathcal{C}}{\partial u_i}   \cdot \frac{\partial   }{\partial u_j} \left(  \mathbf{J} \mathcal{C} \right)
=  \mathbf{J} \frac{\partial   \mathcal{C}}{\partial u_i}   \cdot  \mathbf{J} \frac{\partial   }{\partial u_j} \left(  \mathbf{J} \mathcal{C} \right)
=  - \frac{\partial  }{\partial u_i}  \left( \mathbf{J} \mathcal{C} \right) \cdot \frac{\partial   \mathcal{C}  }{\partial u_j}
\]
and
\[
 \frac{\partial   \mathcal{C}}{\partial u_i}   \cdot \frac{\partial   }{\partial u_i} \left(  \mathbf{J} \mathcal{C} \right)=0,
\]
we see that the second sum also vanishes:  
\[
\sum_{1 \leq i < j \leq 2N} \;   \sqrt{  \mathbf{g} \, } \;   {\mathbf{g}}^{ij}  \left[  \frac{\partial   \mathcal{C}}{\partial u_j}   \; \cdot \frac{\partial   }{\partial u_i} \left(  \mathbf{J} \mathcal{C} \right) +  \frac{\partial   \mathcal{C}}{\partial u_i}   \; \cdot \frac{\partial   }{\partial u_j} \left(  \mathbf{J} \mathcal{C} \right)  \right] =0.
\]
\textbf{(d)} From  $\frac{\partial  \mathcal{D}}{\partial u_{i}} \cdot \mathbf{J}  \mathcal{C} = - \frac{\partial  \mathcal{C}}{\partial u_{i}} \cdot \mathbf{J}  \mathcal{D}$ and  \textbf{(a)}, we have 
\begin{eqnarray*}
&&   \sum_{1 \leq i,j \leq 2N}   {\mathbf{g}}^{ij}  w_{j}  \frac{\partial }{\partial u_{i}}  \left(  {  \mathcal{D} \cdot \mathbf{J}  \mathcal{C}     }   \right)    \\
&=& \sum_{1 \leq i,j \leq 2N}   {\mathbf{g}}^{ij}  w_{j}     \left(   \frac{\partial  \mathcal{D}}{\partial u_{i}} \cdot \mathbf{J}  \mathcal{C}   \right)  +    \mathcal{D}  \cdot \mathbf{J}  \left( \sum_{1 \leq i,j \leq 2N}  {\mathbf{g}}^{ij}  w_{j}  \frac{\partial \mathcal{C}  }{\partial u_{i}}  \right) \\
&=& - \sum_{1 \leq i,j \leq 2N}   {\mathbf{g}}^{ij}  w_{j}     \left(   \frac{\partial  \mathcal{C}}{\partial u_{i}} \cdot \mathbf{J}  \mathcal{D}   \right)  -   \mathbf{J}  \mathcal{D}  \cdot \left( \sum_{1 \leq i,j \leq 2N}  {\mathbf{g}}^{ij}  w_{j}  \frac{\partial \mathcal{C}  }{\partial u_{i}}  \right) \\
&=& - 2   \left( \sum_{1 \leq i,j \leq 2N}   {\mathbf{g}}^{ij}  w_{j}      \frac{\partial  \mathcal{C}}{\partial u_{i}}   \right)   \cdot \mathbf{J}  \mathcal{D}   \\
&=& -2 \, \left[ \;   \mathbf{J} \mathcal{C} -  \left(\mathcal{D} \cdot \mathbf{J}  \mathcal{C} \right) \mathcal{D} \; \right]   \cdot \mathbf{J}  \mathcal{D}  \\
&=& - 2  \,   \mathcal{C} \cdot  \mathcal{D} \\
&=& 0.
\end{eqnarray*}
\textbf{(e)} Recall that the symmetric matrix $\left(  {\mathbf{g}}_{ij}  \right)_{1 \leq i,j \leq 2N}$ is a block 
matrix such that  ${\mathbf{g}}_{ij}={\mathbf{g}}_{ji}=0$ for $i \leq N$ and $j \geq N+1$, and note that  
\[
\frac{\partial }{\partial u_{i}}  \left(  {  \mathcal{D} \cdot \mathbf{J}  \mathcal{C}     }   \right)=
\begin{cases} 
\;\;2 w_{i}, \quad i  \in \left\{1, \cdots, N  \right\},  \\
-2 w_{i}, \;\;\;   i \in \left\{N+1, \cdots, 2N  \right\}.
\end{cases}
\]
Using the second identity in \textbf{(a)}, we have 
\begin{eqnarray*}
&&    \sum_{1 \leq i,j \leq 2N}   {\mathbf{g}}^{ij}     \frac{\partial }{\partial u_{i}}  \left(  {  \mathcal{D} \cdot \mathbf{J}  \mathcal{C}     }   \right) \frac{\partial \mathcal{C}}{\partial u_{j}}  \\
&=&   2  \sum_{1 \leq i,j \leq N}   {\mathbf{g}}^{ij} w_{i}  \frac{\partial  \mathcal{C}}{\partial u_{j}} - 2   \sum_{N+1 \leq i,j \leq 2N}   {\mathbf{g}}^{ij} w_{i}  \frac{\partial  \mathcal{C}}{\partial u_{j}} \\
&=& 2  \sum_{1 \leq i,j \leq 2N}   {\mathbf{g}}^{ij}   w_{i} \frac{\partial \mathcal{D}}{\partial u_{j}} \\
&=& - 2 \left(\, \mathbf{J} \mathcal{D} +  \left(\mathcal{D} \cdot \mathbf{J}  \mathcal{C} \right) \, \mathcal{C} \, \right).
\end{eqnarray*}
\end{proof}

\subsection{Multi-rays cones over the submaninfolds in a sphere}  \label{section multi rays}

\begin{definition}[\textbf{$L$-rays cone in ${\mathbb{R}}^{L(N+1)}$  over a submanifold in ${\mathbb{S}}^{N} \subset {\mathbb{R}}^{N+1}$}]  \label{general multi rays}
Given a submanifold $\Sigma$ in the unit hypersphere  ${\mathbb{S}}^{N} \subset {\mathbb{R}}^{N+1}$, we introduce
the $L$-rays cone in ${\mathbb{R}}^{L(N+1)}$ (possibly with a singularity at the origin)   
\[
  {\mathcal{C}}_{L} \left( \Sigma \right) = 
 \left\{\; \begin{bmatrix} \, r_{1}  \mathbf{P} \, \\   \vdots \\ \, r_{L}   \mathbf{P} \, 
\end{bmatrix}   \in {\mathbb{R}}^{L(N+1)}  \; \vert \;   r_{1}, \cdots, r_{L} \in \mathbb{R}, \mathbf{P} \in \Sigma  \; \right\}.
\] 
\end{definition}

\begin{theorem}[\textbf{Takahashi type equivalence for multi-rays cones}]   \label{ET}
Let ${{\Sigma}^{n}}$ be a submanifold of the unit hypersphere  ${\mathbb{S}}^{N} \subset {\mathbb{R}}^{N+1}$.
Then the following three statements are equivalent:
\begin{enumerate}
\item[\textbf{(a)}]   ${{\Sigma}^{n}}$ is minimal in  ${\mathbb{S}}^{N}$.
\item[\textbf{(b)}]  The following submanifold ${\mathcal{S}}_{L} \left({{\Sigma}^{n}} \right)$ is minimal in  ${\mathbb{S}}^{L(N+1) -1}$.
\[
  {\mathcal{S}}_{L} \left( {{\Sigma}^{n}} \right) = 
 \left\{\; \begin{bmatrix} \, x_{1}  \mathbf{P} \, \\   \vdots \\ \, x_{L}   \mathbf{P} \, 
\end{bmatrix}   \in {\mathbb{S}}^{L(N+1) -1} \subset {\mathbb{R}}^{L(N+1)}  \; \vert \;  \begin{bmatrix} \, x_{1}   \, \\   \vdots \\ \, x_{L}   \, 
\end{bmatrix}  \in   {\mathbb{S}}^{L-1} \subset {\mathbb{R}}^{L},   \mathbf{P} \in \Sigma  \; \right\}.
\] 
\item[\textbf{(c)}]  The $L$-rays cone ${\mathcal{C}}_{L} \left( {{\Sigma}^{n}} \right)$ is minimal in ${\mathbb{R}}^{L(N+1)}$. 
\end{enumerate}
\end{theorem}

\begin{remark}
 The case  $L=1$ of Theorem \ref{ET} was proved in \cite[Theorem 3]{Takahashi1966} and \cite[Proposition 6.1.1]{Simons1968}.
\end{remark}

\begin{proof} 
Since ${\mathcal{C}}_{L} \left({{\Sigma}^{n}} \right)$ is  the cone over ${\mathcal{S}}_{L} \left({{\Sigma}^{n}} \right)$, Takahashi's Theorem \cite{Takahashi1966}
guarantees the equivalence \textbf{(b)} $\Leftrightarrow$ \textbf{(c)}. It now remains to prove the equivalence \textbf{(a)} $\Leftrightarrow$ \textbf{(c)}.
Let ${\mathbf{F}}\left(u_{1}, \cdots, u_{n}\right)$ be a local patch of  ${{\Sigma}^{n}}$ in ${\mathbb{S}}^{N} \subset {\mathbb{R}}^{N+1}$. We write 
\[
g = g_{\Sigma} = {(g_{ij})}_{1 \leq i,j \leq n}={\left(  \frac{\partial \mathbf{F}}{\partial u_{i}}  \cdot   \frac{\partial \mathbf{F}}{\partial u_{j}}    \right)}_{  1 \leq i,j \leq n }.
\]
It follows from Takahashi's Theorem \cite{Takahashi1966} that ${{\Sigma}^{n}}$ becomes a minimal submanifold  in ${\mathbb{S}}^{N}$ if and only if
\[
  {\mathbf{0}}_{  {\mathbb{R}}^{N+1}  } = n{\mathbf{F}} + {\triangle}_{g} {\mathbf{F}}.
\]
Next, the induced patch of the $L$-rays cone ${\mathcal{C}}_{L} \left( {{\Sigma}^{n}} \right)$ reads
\[
\Phi \left(r_{1}, \cdots, r_{L}, u_{1}, \cdots, u_{n} \right) =  
\begin{bmatrix} r_{1}  {\mathbf{F}}\left(u_{1}, \cdots, u_{n}\right)  \\   \vdots \\ r_{L}  {\mathbf{F}}\left(u_{1}, \cdots, u_{n}\right)
\end{bmatrix}  \in {\mathbb{R}}^{L(N+1)}.
\]
By observing that the induced metric of the $L$-rays cone ${\mathcal{C}}_{L} \left( {{\Sigma}^{n}} \right)$ reads
\[
G =  {dr_{1}}^{2} + \cdots +  {dr_{L}}^{2} +  \mathcal{R} g, \quad 
 \mathcal{R}  = {r_{1}}^{2} + \cdots +  {{r}_{L}}^{2},  
\]
we are able to explicitly compute   the mean curvature vector $\mathbf{H}$ of  ${\mathcal{C}}_{L} \left( {{\Sigma}^{n}} \right) \subset 
{\mathbb{R}}^{L(N+1)}$ in terms of local coordinates $u_{1}, \cdots, u_{n}, r_{1}, \cdots, r_{L}:$
\[
\mathbf{H}= {\triangle}_{G} \Phi = \begin{bmatrix} 
  \; \frac{r_{1}}{ \mathcal{R}   }  \left( n{\mathbf{F}} + {\triangle}_{g} {\mathbf{F}} \right) \; \\   \vdots \\  
 \; \frac{r_{L}}{ \mathcal{R}   }  \left( n{\mathbf{F}} + {\triangle}_{g} {\mathbf{F}} \right) \; 
\end{bmatrix}. 
\]
Therefore, we see that the mean curvature vector filed $\mathbf{H}$ of the cone   ${\mathcal{C}}_{L} \left( {{\Sigma}^{n}} \right)$  vanishes  if and only if 
${\mathbf{0}}_{  {\mathbb{R}}^{N+1}  } = n{\mathbf{F}} + {\triangle}_{g} {\mathbf{F}}$.
\end{proof}

  The $L$-rays cone in  ${\mathbb{R}}^{L(2N+2)}$ over the Clifford torus $\frac{1}{\sqrt{2}}  {\mathbb{S}}^{N} \times   \frac{1}{\sqrt{2}}  {\mathbb{S}}^{N}$  
  will play an important role in Theorem \ref{gHEL2}.

\begin{definition}[\textbf{$L$-rays Clifford cone in Euclidean space ${\mathbb{R}}^{L(2N+2)}$}]  \label{GCC}
Let $L \geq 1$, $N \geq 0$ be integers.  We introduce the $L$-rays Clifford cone    in ${\mathbb{R}}^{L(2N+2)}={\mathbb{C}}^{(N+1)L}$. 
 Definition \ref{general multi rays} gives an explicit expression 
\[
 {\mathcal{C}}_{L} \left(   \frac{1}{\sqrt{2}}  {\mathbb{S}}^{N} \times 
 \frac{1}{\sqrt{2}}  {\mathbb{S}}^{N}             \right) = \left\{\; \begin{bmatrix} r_{1} \left( X+ i Y \right) \\   \vdots \\ r_{L}  \left( X+ i Y \right)
\end{bmatrix}    \; \vert \;   r_{1}, \cdots, r_{L} \in \mathbb{R}, {\Vert X \Vert}_{  {\mathbb{R}}^{N+1} } =  {\Vert Y \Vert}_{  {\mathbb{R}}^{N+1} } =1 \; \right\}.
\]
\end{definition}

\begin{cor}[\textbf{Minimality of multi-rays Clifford cones}]     \label{minimal multi-lays clifford cones}
The $L$-rays Clifford cone 
${\mathcal{C}}_{L} \left(   \frac{1}{\sqrt{2}}  {\mathbb{S}}^{N} \times   \frac{1}{\sqrt{2}}  {\mathbb{S}}^{N}             \right)$
is a minimal submanifold in ${\mathbb{R}}^{L(2N+2)}$.
\end{cor}

\begin{remark} We observe that ${\mathcal{C}}_{1} \left(   \frac{1}{\sqrt{2}}  {\mathbb{S}}^{N} \times   \frac{1}{\sqrt{2}}  {\mathbb{S}}^{N}  \right) $ is the classical Clifford cone. After applying reflections in ${\mathbb{R}}^{4N+4}$, ${\mathcal{C}}_{2} \left(   \frac{1}{\sqrt{2}}  {\mathbb{S}}^{N} \times   \frac{1}{\sqrt{2}}  {\mathbb{S}}^{N}  \right) $ is congruent to Harvey-Lawson's \textit{twisted normal cone}  \cite[Theorem 3.17]{HL1982} over the Clifford torus. See  
Example \ref{E4} in Section \ref{MAINresults} and \cite[Example 3.22]{HL1982}. 
\end{remark}

\section{Generalized helicoids in Euclidean space ${\mathbb{R}}^{L(2N+2)+1}$}      \label{MAINresults}

For any given pair $\left( {\lambda}_{0},  {\lambda}_{1} \right) $ of real constants, the submanifold
\[
   \left\{ \; 
 \begin{bmatrix} \; \mathbf{r} \; e^{ \, i  \left( {\lambda}_{1}  \Theta \right) \, }  \\  {\lambda}_{0}  \Theta  \end{bmatrix} \in 
{\mathbb{C}} \times \mathbb{R} \; \vert \;   \Theta, \mathbf{r} \in \mathbb{R}
\right\}
\]
is minimal in ${\mathbb{R}}^{3}$. We present two generalizations and four examples.

\begin{theorem}[\textbf{Sweeping out $L$ indpendent Clifford cones in ${\mathbb{R}}^{2N+2}$}] \label{gHEL1}
Let $L \geq 1$, $N \geq 0$ be integers. Given an  $(L+1)$-tuple ${\Lambda}= \left( {\lambda}_{0}, {\lambda}_{1}, \cdots, {\lambda}_{L} \right)$ of real numbers 
and a collection $\mathbf{C}=\left\{ {\mathcal{C}}_{1}, \cdots, {\mathcal{C}}_{L} \right\}$ of  $L$ independent  
$2N$-dimensional Clifford tori lying in the unit hypersphere ${\mathbb{S}}^{2N+1} \subset  {\mathbb{R}}^{2N+2}$,  we define the generalized helicoid ${\mathcal{H}}^{\Lambda, \mathbf{C}} \subset {\mathbb{R}}^{L(2N+2)+1} = {\mathbb{C}}^{(N+1)L} \times {\mathbb{R}}$ 
\[
   {\mathcal{H}}^{\Lambda, \mathbf{C}} = \left\{ \; \begin{bmatrix}
\, r_{1} \, e^{i( \lambda_{1} \Theta )} \left( {X_{1}} + i  {Y_{1}}   \right) \\ 
\vdots \\
\, r_{L} \, e^{i( \lambda_{L} \Theta )} \left( {X_{L}} + i  {Y_{L}}   \right) \\ 
  {\lambda_{0}} \Theta
\end{bmatrix}  \, \vert \;  \Theta \in \mathbb{R},  r_{t} \in \mathbb{R}, \begin{bmatrix}
X_t \\ 
Y_t
\end{bmatrix} \in {\mathcal{C}}_{t},  1 \leq t \leq L \right\}.
\]
Then,  ${\mathcal{H}}^{\Lambda, \mathbf{C}}$ is a minimal submanifold in  ${\mathbb{R}}^{L(2N+2)+1}$ and invariant under the multi-screw motion  ${\mathcal{S}}_{{\Lambda}}$ 
introduced in Definition \ref{MSM}. 
\end{theorem}

\begin{remark}  \label{CONE by THM 2}
When ${\lambda}_{0}=0$,  ${\mathcal{H}}^{\Lambda, \mathbf{C}}$ becomes a minimal cone in ${\mathbb{R}}^{L(2N+2)}$.  In the particular case when $\left(  {\lambda}_{0},  {\lambda}_{1}, \cdots,  {\lambda}_{L}  \right)=\left(0, \cdots, 0 \right)$,  ${\mathcal{H}}^{\Lambda, \mathbf{C}}$ is the product of $L$ independent Clifford cones.
 \end{remark}

\begin{theorem}[\textbf{Sweeping out the $L$-rays Clifford cone  in ${\mathbb{R}}^{L(2N+2)}$}]   \label{gHEL2}
Let $L \geq 1$, $N \geq 0$ be integers. Given two real constants ${\lambda}_{0}, {\lambda}$ and the $L$-rays Clifford cone ${\mathcal{C}}_{L} \left(   \frac{1}{\sqrt{2}}  {\mathbb{S}}^{N} \times   \frac{1}{\sqrt{2}}  {\mathbb{S}}^{N}  \right)$, 
we define the generalized helicoid $ {\mathcal{H}}^{\lambda, {\lambda}_{0}, L,N}$ in Euclidean space  ${\mathbb{R}}^{L(2N+2)+1} = {\mathbb{C}}^{(N+1)L} \times {\mathbb{R}}$ 
\[
   {\mathcal{H}}^{\lambda, {\lambda}_{0}, L, N} = \left\{ \; \begin{bmatrix}
 e^{i \left( \lambda \Theta \right) }  Z \\ 
  {\lambda}_{0} \Theta
\end{bmatrix} \in {\mathbb{C}}^{(N+1)L} \times {\mathbb{R}} \, \vert \;   \Theta \in \mathbb{R},  Z \in {\mathcal{C}}_{L} \left(   \frac{1}{\sqrt{2}}  {\mathbb{S}}^{N} \times   \frac{1}{\sqrt{2}}  {\mathbb{S}}^{N}  \right)  \right\}. 
\]
More explicitly, we have 
\[
{\mathcal{H}}^{\lambda, {\lambda}_{0}, L, N} = \left\{ \; \begin{bmatrix}
\, r_{1} \, e^{i( \lambda  \Theta )} \left( {X} + i  {Y}   \right) \\ 
\vdots \\
\, r_{L}  \, e^{i( \lambda  \Theta )} \left( {X} + i  {Y}   \right) \\ 
  {\lambda_{0}} \Theta
\end{bmatrix}  \, \vert \;  \Theta \in \mathbb{R},  r  \in \mathbb{R}, \begin{bmatrix}
X  \\ 
Y 
\end{bmatrix} \in   \frac{1}{\sqrt{2}}  {\mathbb{S}}^{N} \times   \frac{1}{\sqrt{2}}  {\mathbb{S}}^{N} \right\}.
\]
Then, the variety  $ {\mathcal{H}}^{\lambda, {\lambda}_{0}, L,N}$ is minimal in  ${\mathbb{R}}^{L(2N+2)+1}$. It is invariant under the multi-screw motion  ${\mathcal{S}}_{{\Lambda}=\left(  \lambda, \cdots, \lambda, {\lambda}_{0} \right)}$ introduced in Definition \ref{MSM}. 
\end{theorem}
 
\begin{remark}  \label{CONE by THM 3}
When $\lambda={\lambda}_{0}=0$,   $ {\mathcal{H}}^{0, 0, L,N}$ is a minimal cone in Corollary \ref{minimal multi-lays clifford cones}.
 \end{remark}
 
 \newpage

\begin{example}[\textbf{Choe-Hoppe's minimal hypersurface \cite[Theorem 2]{CH2013}}]    \label{E1}
Taking $L=1$ in Theorem  \ref{gHEL1} or  Theorem \ref{gHEL2} recovers the Choe-Hoppe helicoid \cite{CH2013}. 
More explicitly, sweeping out the Clifford cone  in  ${\mathbb{R}}^{2N} \subset {\mathbb{R}}^{2N+1}$ 
\[
   {\mathbf{C}}^{2N-1} = \left\{ \; 
\begin{bmatrix}
  \, p_{1}   \, \\ 
  \, q_{1}   \,   \\ 
\vdots \\
  \, p_{N}    \,  \\ 
  \, q_{N}  \,   \\ 
\end{bmatrix}  \in   {\mathbb{R}}^{2N}  \; \; \vert \; \;   {p_{1}}^{2} + \cdots +  {p_{N}}^{2}  = {q_{1}}^{2} + \cdots +  {q_{N}}^{2}  \;
\right\}
\]
yields the minimal submanifold in  ${\mathbb{R}}^{2N+1}$
\[
{\mathcal{H}}_{\lambda}  = \left\{ \; 
\begin{bmatrix}
   x_{1} \\ 
  y_{1}  \\ 
\vdots \\
 x_{N} \\ 
 y_{N}   \\ 
z \\
\end{bmatrix} =
\begin{bmatrix}
  \, p_{1} \cos \Theta  - q_{1} \sin \Theta  \, \\ 
  \, q_{1} \cos \Theta  +p _{1} \sin \Theta  \,   \\ 
\vdots \\
  \, p_{N} \cos \Theta  - q_{N} \sin \Theta   \,  \\ 
  \, q_{N} \cos \Theta  +p _{N} \sin \Theta   \,    \\ 
  \lambda  \, \Theta \\
\end{bmatrix}  \in   {\mathbb{R}}^{2N+1}  \; \vert \;    \Theta \in \mathbb{R}, \; \begin{bmatrix}
  \, p_{1}   \, \\ 
  \, q_{1}   \,   \\ 
\vdots \\
  \, p_{N}    \,  \\ 
  \, q_{N}  \,   \\ 
\end{bmatrix}  \in {\mathbf{C}}^{2N-1} \;
\right\},
\]
for any pitch constant $\lambda \in \mathbb{R}$. 
\end{example}

\begin{remark}  \label{OTHERproofs}
Observing that ${\mathbf{C}}^{2N-1} $ is a cone, one finds that ${\mathcal{H}}_{\lambda}$ is homothetic to ${\mathcal{H}}_{1}$ for any non-zero constant $\lambda \in \mathbb{R}$. 
Up to homotheties, the Choe-Hoppe helicoid in  ${\mathbb{R}}^{2N+1}$ can be represented as the hypersurface
\[
 z= \mathbf{arg} \, \left(  \sqrt{  \; (x_{1}+i y_{1})^{2}  + \cdots + 
  (x_{N}+i y_{N})^{2}  \;  } \;  \right).
\]
We can also deduce its minimality by checking that the function
\[
 f \left( x_{1}, y_{1}, \cdots, x_{N}, y_{N} \right) = \frac{1}{2} \, \mathrm{arctan} \; \left( \, \frac{ 2 x_{1}y_{1} + \cdots + 2 x_{N}y_{N} }{ {x_{1}}^{2} - {y_{1}}^{2} + \cdots +  {x_{N}}^{2} - {y_{N}}^{2}   } \, \right)
\]
satisfies the minimal hypersurface equation in  ${\mathbb{R}}^{2N+1}:$
\[
0 =  \sum_{k=1}^{N} \left[  \frac{\partial }{\partial x_{k}} \left(  \frac{  f_{x_{k}}  }{ {\mathcal{W}} }  \right) +  \frac{\partial }{\partial y_{k}} \left(   \frac{  f_{y_{k}}  }{    {\mathcal{W}} }   \right) \right], \;\;  {\mathcal{W}}=
 \sqrt{  1 +  \sum_{k=1}^{N}  \left( {f_{x_{k}}}^{2} +  {f_{y_{k}}}^{2}  \right)\, }.
\]
\end{remark}

\begin{remark}
There are at least two geometric proofs of the minimality of the classical helicoids in  ${\mathbb{R}}^{3}$, which exploits symmetries of helicoids. For instance, see \cite[Section 2.2]{FT1991} and Karsten's lecture note \cite[Section 2.2]{KGBlec}. Interested readers may also try to give new proofs of the minimality of the Choe-Hoppe helicoid, which extend such geometric arguments.  
\end{remark}

 \newpage

\begin{example}[\textbf{Barbosa-Dajczer-Jorge's helicoids  \cite{BDJ1984}}]    \label{E2}
The case $N=0$ in Theorem  \ref{gHEL1} recovers Barbosa-Dajczer-Jorge's ruled minimal submanifolds \cite{BDJ1984}. More explicitly, given an $(L+1)$-tuple  ${\Lambda}= \left( {\lambda}_{0}, {\lambda}_{1}, \cdots, {\lambda}_{L} \right)$,
by sweeping out an $L$-dimensional plane, we have a minimal submanifold 
\[
{\mathcal{H}}_{\Lambda}  = \left\{ \; 
\begin{bmatrix}
 \; r_{1} \; \cos \left( \lambda_{1} \Theta   \right)   \;  \\ 
 \; r_{1} \;  \sin \left( \lambda_{1} \Theta   \right)   \;  \\ 
\vdots \\
 \;  r_{L} \; \cos \left( \lambda_{L} \Theta   \right)   \;  \\ 
 \; r_{L} \;  \sin \left( \lambda_{L} \Theta   \right)   \;  \\ 
  \; \quad \quad \;      {\lambda_{0}} \, \Theta \\
\end{bmatrix}  \in   {\mathbb{R}}^{2L+1}  \; \vert \;    \Theta, r_{1}, \cdots, r_{L} \in \mathbb{R} \;
\right\}.
\]
Bryant \cite{B1991} proved that it is an austere submanifold, which means that, 
for any normal vector, the set of eigenvalues of its induced shape operator is invariant under the multiplication by $-1$.
These submanifolds can be characterized by two uniqueness results. See  \cite[Theorem 3.10]{BDJ1984} and \cite[Theorem 3.1]{B1991}. Notice that, for $\lambda_{0}=0$, they become minimal cones in  ${\mathbb{R}}^{2L}$. 
\end{example}

\begin{example}[\textbf{Minimal submanifolds in the unit sphere ${\mathbb{S}}^{L(2N+2)-1}$}]    \label{E3}
Taking ${\lambda}_{0}=0$ in Theorem  \ref{gHEL1}, we obtain a minimal cone in ${\mathbb{R}}^{L(2N+2)}$ given by
\[
   {\mathcal{H}}^{\lambda_{1}, \cdots, \lambda_{L}, \mathbf{C}} = \left\{ \; \begin{bmatrix}
\, r_{1} e^{i( \lambda_{1} \Theta )} \left( {X_{1}} + i  {Y_{1}}   \right) \\ 
\vdots \\
\, r_{L} e^{i( \lambda_{L} \Theta )} \left( {X_{L}} + i  {Y_{L}}   \right) \\ 
\end{bmatrix}  \, \vert \;  \Theta \in \mathbb{R},  r_{t} \in \mathbb{R}, \begin{bmatrix}
X_t \\ 
Y_t
\end{bmatrix} \in {\mathcal{C}}_{t},  1 \leq t \leq L \right\},
\]
where  $\mathbf{C}=\left\{ {\mathcal{C}}_{1}, \cdots, {\mathcal{C}}_{L} \right\}$ denotes a collection of $L$ independent $2N$-dimensional Clifford tori in ${\mathbb{S}}^{2N+1} \subset  {\mathbb{R}}^{2N+2}$. The fact that $ {\mathcal{H}}^{\lambda_{1}, \cdots, \lambda_{L}, \mathbf{C}}$ is a minimal cone in  ${\mathbb{R}}^{L(2N+2)}$ guarantees that the intersection
$\Sigma = {\mathcal{H}}^{\lambda_{1}, \cdots, \lambda_{L}, \mathbf{C}} \cap {\mathbb{S}}^{L(2N+2)-1}$
becomes a minimal submanifold in the unit sphere ${\mathbb{S}}^{L(2N+2)-1} \subset {\mathbb{R}}^{L(2N+2)}$. More explicitly, 
we have 
 \[
   \Sigma= \left\{ \; \begin{bmatrix}
\, p_{1} e^{i( \lambda_{1} \Theta )} \left( {X_{1}} + i  {Y_{1}}   \right) \\ 
\vdots \\
\, p_{L} e^{i( \lambda_{L} \Theta )} \left( {X_{L}} + i  {Y_{L}}   \right) \\ 
\end{bmatrix}  \, \vert \;  \Theta \in \mathbb{R},  \sum_{t=1}^{L}  {p_{t}}^{2}=1,  \begin{bmatrix}
X_t \\ 
Y_t
\end{bmatrix} \in {\mathcal{C}}_{t},  1 \leq t \leq L \right\}.
\]
In the particular case when $(L,N)=(2,0)$, we recover the family of Lawson's ruled minimal surfaces \cite[Section 7]{Lawson1970}  in ${\mathbb{S}}^{3} \subset {\mathbb{R}}^{4}$:
\[
   \Sigma= \left\{ \; \begin{bmatrix}
\, \cos t \cos  ( \lambda_{1} \Theta )    \\  
\, \cos t \sin  ( \lambda_{1} \Theta )    \\  
\,  \sin t \cos  ( \lambda_{2} \Theta )   \\  
\,   \sin t  \sin  ( \lambda_{2} \Theta )  \\  
\end{bmatrix}  \in  {\mathbb{R}}^{4}   \, \vert \;  t, \Theta \in \mathbb{R} \; \right\},
\]
for any pair  $( {\lambda}_{1}, {\lambda}_{2} ) \neq (0,0)$ of real constants. 
\end{example}

\newpage

\begin{example}[\textbf{Harvey-Lawson's volume-minimizing cone in ${\mathbb{R}}^{4N+4}$ \cite{HL1982}}]    \label{E4}
 Corollary \ref{minimal multi-lays clifford cones} with $L=2$ or Theorem  \ref{gHEL2} with ${\lambda}_{0}={\lambda}=0$ and $L=2$  recover Harvey-Lawson's twisted normal 
cone \cite[Example 3.22]{HL1982} over the Clifford torus  $ \frac{1}{\sqrt{2}}  {\mathbb{S}}^{N} \times   \frac{1}{\sqrt{2}}  {\mathbb{S}}^{N} \subset {\mathbb{S}}^{2N+1}$. According to \cite[Theorem 3.17]{HL1982}, the austerity of the Clifford torus in ${\mathbb{S}}^{2N+1}$ guarantees that the  
cone  
\[
{\Sigma}^{2N+2} = \left\{ \; 
\begin{bmatrix}
  \, r_{1} X   \, \\ 
  \, r_{1} Y  \,   \\ 
  \, r_{2} X    \,  \\ 
  \, r_{2} Y  \,   \\ 
\end{bmatrix}  \in   {\mathbb{R}}^{4N+4}  \; \vert  \;  r_{1}, r_{2} \in \mathbb{R}, \;   {\Vert X \Vert}_{  {\mathbb{R}}^{N+1} } =1, \; {\Vert Y \Vert}_{  {\mathbb{R}}^{N+1} } =1   \;
\right\}
\]
is homologically volume minimizing.
\end{example}

\section{Proof of main results}

We present details of the 
proof of  Theorem  \ref{gHEL1}, which exploits five identities in Lemma \ref{magic}. Since the proof of Theorem  \ref{gHEL2} is similar, we shall omit it.  

Our aim is to show that the generalized helicoid ${\mathcal{H}}^{\Lambda, \mathbf{C}}$ is minimal in ${\mathbb{R}}^{L(2N+2)+1}$. 
In the particular case when $\left(  {\lambda}_{0},  {\lambda}_{1}, \cdots,  {\lambda}_{L}  \right)=\left(0, \cdots, 0 \right)$, it becomes the 
product of $L$ independent Clifford cones. From now on, we assume that $\left(  {\lambda}_{0},  {\lambda}_{1}, \cdots,  {\lambda}_{L}  \right) \neq \left(0, \cdots, 0 \right)$.

For each index $s \in \left\{ 1, \cdots, L \right\}$, let $C^{s} \left( u^{s}_{1}, \cdots, u^{s}_{2N} \right) $ denote a local chart of the Clifford tori $\frac{1}{\sqrt{2}}  {\mathbb{S}}^{N} \times   \frac{1}{\sqrt{2}}  {\mathbb{S}}^{N} $  in ${\mathbb{S}}^{2N+1} \subset {\mathbb{R}}^{2N+2}={\mathbb{C}}^{N+1}$. These induce a local 
patch $\mathbf{F}$ of the generalized helicoid ${\mathcal{H}}^{\Lambda, \mathbf{C}} \subset {\mathbb{R}}^{L(2N+2)+1}$ 
\[
\mathbf{F} \left( u^{1}_{1}, \cdots, u^{1}_{2N}, \cdots, u^{L}_{1}, \cdots, u^{L}_{2N}, \Theta, r_{1}, \cdots, r_{L}  \right)
= 
\begin{bmatrix}
\, r_{1} e^{i( \lambda_{1} \Theta )}  C^{1} \left( u^{1}_{1}, \cdots, u^{1}_{2N} \right) \\ 
\vdots \\
\, r_{L} e^{i( \lambda_{L} \Theta )}  C^{L} \left( u^{L}_{1}, \cdots, u^{L}_{2N} \right) \\ 
  {\lambda_{0}} \Theta
\end{bmatrix}.
\]

We will show that the mean curvature vector ${\triangle}_{ G_{{\mathcal{H}}^{\Lambda, \mathbf{C}}} } \mathbf{F}$ vanishes.
Here, ${\triangle}_{ G_{{\mathcal{H}}^{\Lambda, \mathbf{C}}} }$ denote the Laplace-Beltrami operator on ${\mathcal{H}}^{\Lambda, \mathbf{C}}$ induced by the patch $\mathbf{F}$ of the generalized helicoid ${\mathcal{H}}^{\Lambda, \mathbf{C}}$. More explicitly, we need to prove equalities 
\begin{enumerate}
\item[\textbf{(a)}]
\[
 {\triangle}_{ G_{{\mathcal{H}}^{\Lambda, \mathbf{C}}} }  \left(  {\lambda_{0}} \Theta  \right) \equiv 0.
\]
\item[\textbf{(b)}] 
\[
 {\triangle}_{ G_{{\mathcal{H}}^{\Lambda, \mathbf{C}}} }    \left(  r_{t} e^{i( \lambda_{t} \Theta )}  C^{t}   \right) \equiv 0, \quad t \in \left\{ 1, \cdots, L \right\}.
\]
\end{enumerate}
 
\textbf{Step A.} Let ${\left( g^{s}_{ij} \right) }_{ 1 \leq i, j \leq 2N }$ denote the matrix of the 
first fundamental form  induced by the 
patch $C^{s} \left( u^{s}_{1}, \cdots, u^{s}_{2N} \right)$ of the Clifford torus $\frac{1}{\sqrt{2}}  {\mathbb{S}}^{N} \times   \frac{1}{\sqrt{2}}  {\mathbb{S}}^{N}$ lying in ${\mathbb{S}}^{2N+1} \subset {\mathbb{R}}^{2N+2}$. We adopt the notation
\[
{ \mathbf{g}}^{s} := \mathrm{det} \; { \left({g}^{s}_{ij}  \right) }_{1 \leq i, j \leq 2N}.
\]
Then, the induced metric  $G_{{}_{{\mathcal{H}}^{\Lambda, \mathbf{C}}} }$ of  ${\mathcal{H}}^{\Lambda, \mathbf{C}}$ in coordinates $u^{1}_{1}$, $\cdots$, $u^{1}_{2N}$, $\cdots$, $u^{L}_{1}$, $\cdots$, $u^{L}_{2N}$, $\Theta$, $r_{1}$, $\cdots$, $r_{L}$ reads 
\[
        G_{{}_{{\mathcal{H}}^{\Lambda, \mathbf{C}}} } 
 =\;  \sum_{s=1}^{L}  \sum_{1 \leq i,j \leq 2N}   {r_{s}}^{2}  g^{s}_{ij} \, du^{s}_{i} du^{s}_{j}  + 2
        \sum_{s=1}^{L}  \sum_{i=1}^{2N}  {{\lambda}_{s}} {r_{s}}^{2}  w^{s}_{i}  d\Theta du^{s}_{i}  
   +  \,\mathcal{R}  \; {d\Theta}^{2}  + \;  \sum_{t=1}^{L}  {dr_{s}}^{2},  
\]
where we define 
\begin{equation} \label{Rmeans}
\mathcal{R} =  {\lambda_{1}}^{2} {r_{1}}^{2} + \cdots + {\lambda_{L}}^{2} {r_{L}}^{2} + {\lambda_{0}}^{2}>0,  
\end{equation}
and 
\[
w^{s}_{i}=   \frac{\partial C^{s}_{i} }{\partial u^{s}_i} \cdot \mathbf{J} C^{s}_{i}, \quad s \in  \left\{ 1, \cdots, L \right\}, \,
i \in \left\{ 1, \cdots, 2N \right\}.
\]
By using the cofactor expansion of determinant  or the Laplace formula, we compute the determinant
\begin{equation} \label{Gmeans}
 \mathbf{G}:=\mathrm{det}  \left( G_{{}_{{\mathcal{H}}^{\Lambda, \mathbf{C}}} }  \right) 
= \, \mathcal{P}  \;  { \left( r_{1}\cdots r_{L} \right) }^{4N}   \prod_{s=1}^{L} { \mathbf{g}}^{s}, 
\end{equation}
where we have, by \textbf{(b)} of Lemma \ref{magic}, 
\begin{equation} \label{Pmeans}
\mathcal{P}  := \mathcal{R} - \sum_{1 \leq s \leq L} \sum_{1 \leq i,j \leq 2N}   {  {\lambda}_{s} }^{2} {r_{s}}^{2}   {\left( {g}^{s} \right)}^{ij} w^{s}_{i} w^{s}_{j}
=    {\lambda_{0}}^{2} + \sum_{s=1}^{L}    {  {\lambda}_{s} }^{2} {r_{s}}^{2}   \left( D^s  \cdot \mathbf{J} C^{s} \right)^{2} >0. 
\end{equation}
From now on, we work on the points when $\mathbf{G}=\mathrm{det}  \left( G_{{}_{{\mathcal{H}}^{\Lambda, \mathbf{C}}} }  \right) $ does not vanish,
 or equivalently, when none of $r_{1}, \cdots, r_{L}$ vanishes. Write 
\begin{equation} \label{Dmeans}
 d^{s}_{i} = \sum_{j=1}^{2N} {\lambda}_{s}  {\left( {g}^{s} \right)}^{ij} w^{s}_{j}, \quad s \in  \left\{ 1, \cdots, L \right\}, \,
i \in \left\{ 1, \cdots, 2N \right\}.
\end{equation}
Then, the components of $ { \left( G_{{}_{{\mathcal{H}}^{\Lambda, \mathbf{C}}} } \right) }^{-1}$  in the local coordinates $u^{1}_{1}$, $\cdots$, $u^{1}_{2N}$, $\cdots$, $u^{L}_{1}$, $\cdots$, $u^{L}_{2N}$, $\Theta$, $r_{1}$, $\cdots$, $r_{L}$ reads: 
 \[
 G^{ u^{s}_{i}  u^{s}_{j} }  =  \frac{  \,  {\left( {g}^{s} \right)}^{ij} \,  }{ {r_{s}}^{2}  }  +   \frac{ \, d^{s}_{i} \, d^{s}_{j} \, }{ \mathcal{P} }, \quad 
 G^{ u^{s}_{i}  \Theta }  = G^{  \Theta u^{s}_{i}  }    =  \frac{ \, -  d^{s}_{i}  \, }{ \mathcal{P} }, \quad 
 G^{ \Theta  \Theta }  = \frac{1}{ \mathcal{P}  }, \quad 
 G^{ r_{t} r_{t} }  = 1.
\] 

The other components of $ { \left( G_{{}_{{\mathcal{H}}^{\Lambda, \mathbf{C}}} } \right) }^{-1}$ are all zero. Finally, we find the induced Laplace-Beltrami operator with respect to the metric  $G_{{}_{{\mathcal{H}}^{\Lambda, \mathbf{C}}} }.$

\begin{eqnarray*}
 {\triangle}_{ G_{{\mathcal{H}}^{\Lambda, \mathbf{C}}} }
&=& \;\;\;   \frac{1}{\sqrt{\mathbf{G}}}\sum_{s=1}^{L}  \sum_{1 \leq i,j \leq 2N}    \frac{\partial}{\partial  u^{s}_{i} } 
\left(  \sqrt{\mathbf{G}} \,  \frac{  \,  {\left( {g}^{s} \right)}^{ij} \,  }{ {r_{s}}^{2}  }   \;  \frac{\partial}{\partial  u^{s}_{j} } 
  \;\; \right)  \\
&&  + \frac{1}{\sqrt{\mathbf{G}}}   \sum_{s=1}^{L} \sum_{1 \leq i,j \leq 2N}    \frac{\partial}{\partial  u^{s}_{i} } 
\left(  \sqrt{\mathbf{G}}  \, \frac{ \, d^{s}_{i} \, d^{s}_{j} \, }{ \mathcal{P} }   \;  \frac{\partial}{\partial  u^{s}_{j} } 
  \;\; \right) \\
&&  + \frac{1}{\sqrt{\mathbf{G}}}  \sum_{s=1}^{L} \sum_{i=1}^{2N}    \frac{\partial}{\partial  u^{s}_{i} }  
\left(    \sqrt{\mathbf{G}} \,  \frac{ \, -  d^{s}_{i}  \, }{ \mathcal{P} }   \frac{\partial}{\partial  \Theta }  \;\;   \right)   \\
&&  + \frac{1}{\sqrt{\mathbf{G}}}  \sum_{s=1}^{L} \sum_{i=1}^{2N}    \frac{\partial}{\partial  \Theta }
\left(    \sqrt{\mathbf{G}} \,  \frac{ \, -  d^{s}_{i}  \, }{ \mathcal{P} }    \frac{\partial}{\partial  u^{s}_{i} }   \;\;   \right)   \\
&& +\frac{1}{\sqrt{\mathbf{G}}}  \left(    \sqrt{\mathbf{G}}  \;  \frac{1}{  \mathcal{P}   }     \frac{\partial}{\partial  \Theta }  \;\; \right)   \\
&& + \frac{1}{\sqrt{\mathbf{G}}}     \sum_{s=1}^{L}      \frac{\partial}{\partial  r^{s} }  \left( \sqrt{\mathbf{G}} 
  \frac{\partial}{\partial  r^{s} }  \;\;    \right).
\end{eqnarray*}

\textbf{Step B.}  We next show that 
\[
   {\triangle}_{ G_{{\mathcal{H}}^{\Lambda, \mathbf{C}}} } \Theta \equiv 0,
\]
which implies that the last coordinate in ${\mathbb{R}}^{L(2N+2)+1}$ is harmonic on the generalized helicoid ${\mathcal{H}}^{\Lambda, \mathbf{C}}$. 
According to the formula for ${\triangle}_{ G_{{\mathcal{H}}^{\Lambda, \mathbf{C}}} }$ deduced in \textbf{Step A}, it reduces to prove the equality
\[
    \sum_{s=1}^{L} \sum_{i=1}^{2N}    \frac{\partial}{\partial  u^{s}_{i} }  
\left(    \sqrt{\mathbf{G}} \,  \frac{ \,  d^{s}_{i}  \, }{ \mathcal{P} }     \right) = 0.
\]
We claim that, for each fixed $s \in \left\{1, \cdots, L \right\}$, 
\begin{equation} \label{sum02}
\sum_{i=1}^{2N}    \frac{\partial}{\partial  u^{s}_{i} }  
\left(    \sqrt{\mathbf{G}} \,  \frac{ \,  d^{s}_{i}  \, }{ \mathcal{P} }     \right) = 0. 
\end{equation}
According to the equality
\[
 \sqrt{ \mathbf{G} } = \sqrt{  \mathcal{P} }  \;  { \left( r_{1}\cdots r_{L} \right) }^{2N}  \prod_{s=1}^{L}    \sqrt{  { \mathbf{g}}^{s} \; }, 
\]
and the definition (\ref{Dmeans})
\[
d^{s}_{i} = \sum_{j=1}^{2N} {\lambda}_{s}  {\left( {g}^{s} \right)}^{ij} w^{s}_{j},
\]
it is sufficient to check the identity
\begin{equation} \label{sum03}
 \sum_{1 \leq i,j \leq 2N}      \frac{\partial}{\partial  u^{s}_{i} }  
\left(  \frac{  1 }{   \sqrt{ \mathcal{P} \, } }  \;  \sqrt{  { \mathbf{g}}^{s} \, }  {\left( {g}^{s} \right)}^{ij} w^{s}_{j}   \right) = 0. 
\end{equation}
or equivalently,
\[
\frac{1}{ \sqrt{ \mathcal{P} } }  \sum_{1 \leq i,j \leq 2N}   \frac{\partial}{\partial u^{s}_{i}} \left(   \sqrt{  {\mathbf{g}}^{s} \, } \;   {\left( {g}^{s} \right)}^{ij}  w^{s}_{j} \right) +    \sum_{1 \leq i,j \leq 2N}   \sqrt{  {\mathbf{g}}^{s} \, }     {\left( {g}^{s} \right)}^{ij}   w^{s}_{j}     \frac{\partial   }{\partial  u^{s}_{i} }  \left(  \frac{1}{  \sqrt{ \mathcal{P}} } \right)      = 0.
\]
The identity \textbf{(c)} of Lemma \ref{magic} guarantees that the first sum vanishes. To prove that the second sum vanishes, we are required to show
\[
 \sum_{1 \leq i,j \leq 2N}      {\left( {g}^{s} \right)}^{ij}   w^{s}_{j}     \frac{\partial  \mathcal{P}  }{\partial  u^{s}_{i} }  = 0. 
\]
From the definition 
\[
\mathcal{P}  = \mathcal{R} - \sum_{1 \leq s \leq L} \sum_{1 \leq i,j \leq 2N}   {  {\lambda}_{s} }^{2} {r_{s}}^{2}   {\left( {g}^{s} \right)}^{ij} w^{s}_{i} w^{s}_{j},
\]
and the identity \textbf{(b)} of Lemma \ref{magic},
we have 
\[
   \frac{\partial  \mathcal{P}  }{\partial  u^{s}_{i} }  =  -   {  {\lambda}_{s} }^{2} {r_{s}}^{2}  
   \frac{\partial   }{\partial  u^{s}_{i} }  \left( \sum_{1 \leq i,j \leq 2N}    {\left( {g}^{s} \right)}^{ij} w^{s}_{i} w^{s}_{j}   \right)
 =  -   {  {\lambda}_{s} }^{2} {r_{s}}^{2}  
   \frac{\partial   }{\partial  u^{s}_{i} }  \left(           1 - \left(  {  D^s  \cdot \mathbf{J} C^{s}   }   \right)^{2}   \right).
\]
We thus need to prove 
\[
\sum_{1 \leq i,j \leq 2N}     {\left( {g}^{s} \right)}^{ij}   w^{s}_{j}    \frac{\partial   }{\partial  u^{s}_{i} }  \left(            1 - \left(  D^s  \cdot \mathbf{J} C^{s}  \right)^{2} \right) = 0.
\]
So, it is enough to obtain
\[
\sum_{1 \leq i,j \leq 2N}      {\left( {g}^{s} \right)}^{ij}   w^{s}_{j}    \frac{\partial   }{\partial  u^{s}_{i} }  \left(  D^s  \cdot \mathbf{J} C^{s}    \right) = 0.
\]
However, because of the identity \textbf{(d)} of Lemma \ref{magic},  this sum vanishes.
\textbf{Step C.} It now remains to prove that, for each index $t \in \left\{ 1, \cdots, L \right\}$, 
\[
 {\triangle}_{ G_{{\mathcal{H}}^{\Lambda, \mathbf{C}}} }    \left(  r_{t} e^{i( \lambda_{t} \Theta )}  C^{t}   \right) \equiv 0. 
\]
  According to the formula for ${\triangle}_{ G_{{\mathcal{H}}^{\Lambda, \mathbf{C}}} }$ deduced in \textbf{Step A}, it reduces to prove the equality
\begin{eqnarray*}
 0
&=& \;\;\;   \sum_{s=1}^{L}  \sum_{1 \leq i,j \leq 2N}    \frac{\partial}{\partial  u^{s}_{i} } 
\left(  \sqrt{\mathbf{G}} \,  \frac{  \,  {\left( {g}^{s} \right)}^{ij} \,  }{ {r_{s}}^{2}  }   \;  \frac{\partial}{\partial  u^{s}_{j}} 
  \;   \left(  r_{t} e^{i( \lambda_{t} \Theta )}  C^{t}   \right) \; \right)  \\
&&  +    \sum_{s=1}^{L} \sum_{1 \leq i,j \leq 2N}    \frac{\partial}{\partial  u^{s}_{i} } 
\left(  \sqrt{\mathbf{G}}  \, \frac{ \, d^{s}_{i} \, d^{s}_{j} \, }{ \mathcal{P} }   \;  \frac{\partial}{\partial  u^{s}_{j} } 
 \;   \left(  r_{t} e^{i( \lambda_{t} \Theta )}  C^{t}   \right) \;  \right) \\
&&  +    \sum_{s=1}^{L} \sum_{i=1}^{2N}    \frac{\partial}{\partial  u^{s}_{i} }  
\left(    \sqrt{\mathbf{G}} \,  \frac{ \, -  d^{s}_{i}  \, }{ \mathcal{P} }   \frac{\partial}{\partial  \Theta }  \;   \left(  r_{t} e^{i( \lambda_{t} \Theta )}  C^{t}   \right) \;    \right)   \\
&&  +    \sum_{s=1}^{L} \sum_{i=1}^{2N}    \frac{\partial}{\partial  \Theta }
\left(    \sqrt{\mathbf{G}} \,  \frac{ \, -  d^{s}_{i}  \, }{ \mathcal{P} }    \frac{\partial}{\partial  u^{s}_{i} }  \;   \left(  r_{t} e^{i( \lambda_{t} \Theta )}  C^{t}   \right) \;   \right)   \\
&& +   \left(    \sqrt{\mathbf{G}}  \;  \frac{1}{  \mathcal{P}   }     \frac{\partial}{\partial  \Theta } \;   \left(  r_{t} e^{i( \lambda_{t} \Theta )}  C^{t}   \right) \;  \right)   \\
&& +   \sum_{s=1}^{L}      \frac{\partial}{\partial  r^{s} }  \left( \sqrt{\mathbf{G}} 
  \frac{\partial}{\partial  r^{s} } \;   \left(  r_{t} e^{i( \lambda_{t} \Theta )}  C^{t}   \right) \;    \right).
\end{eqnarray*}
We express this equality as the sum
\[
 0 = { \mathbf{S} }_{1}  +{ \mathbf{S} }_{2}  +{ \mathbf{S} }_{3}  +{ \mathbf{S} }_{4}  +{ \mathbf{S} }_{5}  +{ \mathbf{S} }_{6}.
\]
\textbf{Step C1.} 
 We recall that $\mathbf{G} = \, \mathcal{P}  \;  { \left( r_{1}\cdots r_{L} \right) }^{4N}   \prod_{s=1}^{L} { \mathbf{g}}^{s}$. 
We introduce
\[
     Q_{s} = \sqrt{ \; \prod_{s \in \{1, \cdots, L \} - \{ \alpha \} }  \left(   \; {r_{s}}^{4N}  { \mathbf{g}}^{s} \; \right) \; }, \quad s \in  \{1, \cdots, L \}
\]
to get the factorization 
\begin{equation} \label{FAC}
\sqrt{\mathbf{G}} = \sqrt{ \mathcal{P}} \; {r_{s}}^{2N} \sqrt{ { \mathbf{g}}^{s} \, } Q_{s}, \quad s \in  \{1, \cdots, L \}.
\end{equation}
We evaluate the sum  $ { \mathbf{S} }_{1}$. 
\begin{eqnarray*}
 { \mathbf{S} }_{1} &=& \sum_{s=1}^{L}  \sum_{1 \leq i,j \leq 2N}    \frac{\partial}{\partial  u^{s}_{i} } 
\left(  \sqrt{\mathbf{G}} \,  \frac{  \,  {\left( {g}^{s} \right)}^{ij} \,  }{ {r_{s}}^{2}  }   \;  \frac{\partial}{\partial  u^{s}_{j}} 
  \;   \left(  r_{t} e^{i( \lambda_{t} \Theta )}  C^{t}   \right) \; \right) \\
 &=& \;\;\; \quad  \sum_{1 \leq i,j \leq 2N}    \frac{\partial}{\partial  u^{t}_{i} } 
\left(  \sqrt{\mathbf{G}} \,  \frac{  \,  {\left( {g}^{t} \right)}^{ij} \,  }{ {r_{t}}^{2}  }   \;  \frac{\partial}{\partial  u^{t}_{j}} 
  \;   \left(  r_{t} e^{i( \lambda_{t} \Theta )}  C^{t}   \right) \; \right) \\
 &=& \frac{1}{r_{t}} e^{i( \lambda_{t} \Theta )}  \sum_{1 \leq i,j \leq 2N}    \frac{\partial}{\partial  u^{t}_{i} } 
\left(  \sqrt{\mathbf{G}} \,  {\left( {g}^{t} \right)}^{ij} \,    \;  \frac{\partial  C^{t} }{\partial  u^{t}_{j}} 
  \; \right) \\
 &=&   {r_{t}}^{2N-1}  Q_{t}  \; e^{i( \lambda_{t} \Theta )}  \sum_{1 \leq i,j \leq 2N}    \frac{\partial}{\partial  u^{t}_{i} } 
\left(  \sqrt{ \mathcal{P}} \,  \sqrt{ { \mathbf{g}}^{t} \, }  \,  {\left( {g}^{t} \right)}^{ij} \,    \;  \frac{\partial  C^{t} }{\partial  u^{t}_{j}} 
  \; \right)   \\  
 &=&  \quad {r_{t}}^{2N-1}  Q_{t} \sqrt{ \mathcal{P}}  \; e^{i( \lambda_{t} \Theta )}  \sum_{1 \leq i,j \leq 2N}    \frac{\partial}{\partial  u^{t}_{i} } 
\left(   \sqrt{ { \mathbf{g}}^{t} \, }  \,  {\left( {g}^{t} \right)}^{ij} \,    \;  \frac{\partial  C^{t} }{\partial  u^{t}_{j}} 
  \; \right)     \\  
 && + \;  {r_{t}}^{2N-1}  Q_{t}  \; e^{i( \lambda_{t} \Theta )}  \sum_{1 \leq i,j \leq 2N}  
    \,  \sqrt{ { \mathbf{g}}^{t} \, }  \,  {\left( {g}^{t} \right)}^{ij} \,    \frac{\partial \sqrt{ \mathcal{P}}}{\partial  u^{t}_{i} }  \;  \frac{\partial  C^{t} }{\partial  u^{t}_{j}}.
\end{eqnarray*}
We make two observations.
First, as in the proof of the identity in \textbf{(c)} of Lemma \ref{magic},  by using the minimality of the Clifford torus
$\frac{1}{\sqrt{2}}  {\mathbb{S}}^{N} \times   \frac{1}{\sqrt{2}}  {\mathbb{S}}^{N} $  in the unit hypersphere ${\mathbb{S}}^{2N+1} \subset {\mathbb{R}}^{2N+2}$, we can simplify the sum in the first term:
\[
  \sum_{1 \leq i,j \leq 2N}    \frac{\partial}{\partial  u^{t}_{i} } 
\left(   \sqrt{ { \mathbf{g}}^{t} \, }  \,  {\left( {g}^{t} \right)}^{ij} \,    \;  \frac{\partial  C^{t} }{\partial  u^{t}_{j}} \right) =
  \sqrt{ { \mathbf{g}}^{t} \, } \;  {\triangle}_{g^t}  C^t = -  2N  \sqrt{ { \mathbf{g}}^{t} \, }  C^{t},
\]
Second, from the definition (\ref{Pmeans}) and \textbf{(b)} of Lemma \ref{magic}, we have
\begin{eqnarray*}
   \frac{\partial \sqrt{ \mathcal{P} }}{\partial  u^{t}_{i} }  
 &=&  \frac{1}{2 \sqrt{   \mathcal{P}   }} \;  \frac{\partial }{\partial  u^{t}_{i} }  \left(    \mathcal{R} - \sum_{1 \leq s \leq L} \sum_{1 \leq i,j \leq 2N}   {  {\lambda}_{s} }^{2} {r_{s}}^{2}   {\left( {g}^{s} \right)}^{ij} w^{s}_{i} w^{s}_{j}   \right) \\  
&=& -  \frac{ {  {\lambda}_{t} }^{2} {r_{t}}^{2}   }{2 \sqrt{   \mathcal{P}   }}  \; \frac{\partial }{\partial  u^{t}_{i} }  \left( 
 1 - \left(  {  D^t \cdot \mathbf{J}  C^t    }   \right)^{2}    \right) \\  
 &=&   \frac{ {  {\lambda}_{t} }^{2} {r_{t}}^{2}   }{  \sqrt{   \mathcal{P}   }}   \left(    D^t \cdot \mathbf{J}  C^t      \right)     \; \frac{\partial }{\partial  u^{t}_{i} }   \left(    D^t \cdot \mathbf{J}  C^t      \right),   
\end{eqnarray*}
and then, by \textbf{(e)} of Lemma \ref{magic}, 
 \begin{eqnarray*}
&&   \sum_{1 \leq i,j \leq 2N}  
    \,  \sqrt{ { \mathbf{g}}^{t} \, }  \,  {\left( {g}^{t} \right)}^{ij} \,    \frac{\partial \sqrt{ \mathcal{P}}}{\partial  u^{t}_{i} }  \;  \frac{\partial  C^{t} }{\partial  u^{t}_{j}}   \\    
&=&   {  {\lambda}_{t} }^{2} {r_{t}}^{2}    \frac{  \sqrt{ { \mathbf{g}}^{t} \, }  }{  \sqrt{   \mathcal{P}   }}   \left(    D^t \cdot \mathbf{J}  C^t      \right)    \sum_{1 \leq i,j \leq 2N}   \,  {\left( {g}^{t} \right)}^{ij} \,    \frac{\partial }{\partial  u^{t}_{i} }   \left(    D^t \cdot \mathbf{J}  C^t      \right) \;  \frac{\partial  C^{t} }{\partial  u^{t}_{j}}.  \\   
&=&   - 2 {  {\lambda}_{t} }^{2} {r_{t}}^{2}    \frac{  \sqrt{ { \mathbf{g}}^{t} \, }  }{  \sqrt{   \mathcal{P}   }}   \left(    D^t \cdot \mathbf{J}  C^t      \right)    \left[     \mathbf{J} D^t +  \left( D^t \cdot \mathbf{J} C^t \right) \, C^t \,    \right].
\end{eqnarray*}
It follows that
\begin{eqnarray*}
 { \mathbf{S} }_{1}
 &=&    -  \; 2N \,  {r_{t}}^{2N-1}  Q_{t} \sqrt{ \mathcal{P}}  \sqrt{ { \mathbf{g}}^{t} \, }   \; e^{i( \lambda_{t} \Theta )}    C^{t}       \\  
 && - \; 2 {  {\lambda}_{t} }^{2}   {r_{t}}^{2N+1}  Q_{t}  \frac{  \sqrt{ { \mathbf{g}}^{t} \, }  }{  \sqrt{   \mathcal{P}   }}    \left(    D^t \cdot \mathbf{J}  C^t      \right)  \; e^{i( \lambda_{t} \Theta )}   \left[     \mathbf{J} D^t +  \left( D^t \cdot \mathbf{J} C^t \right) \, C^t \,    \right].
\end{eqnarray*}
\textbf{Step C2.} We will use the factorization (\ref{FAC}) obtained in \textbf{Step C1}: 
\[
\sqrt{\mathbf{G}} = \sqrt{ \mathcal{P}} \; {r_{t}}^{2N} \sqrt{ { \mathbf{g}}^{t} \, } Q_{t}. 
\]
We expand the sum  $ { \mathbf{S} }_{2}$:
\begin{eqnarray*}
  { \mathbf{S} }_{2} &=&  \sum_{s=1}^{L} \sum_{1 \leq i,j \leq 2N}    \frac{\partial}{\partial  u^{s}_{i} } 
\left(  \sqrt{\mathbf{G}}  \, \frac{ \, d^{s}_{i} \, d^{s}_{j} \, }{ \mathcal{P} }   \;  \frac{\partial}{\partial  u^{s}_{j} } 
 \;   \left(  r_{t} e^{i( \lambda_{t} \Theta )}  C^{t}   \right) \;  \right) \\
 &=& \;\;\; \quad \sum_{1 \leq i,j \leq 2N}    \frac{\partial}{\partial  u^{t}_{i} } 
\left(  \sqrt{\mathbf{G}}  \, \frac{ \, d^{t}_{i} \, d^{t}_{j} \, }{ \mathcal{P} }   \;  \frac{\partial}{\partial  u^{t}_{j} } 
 \;   \left(  r_{t} e^{i( \lambda_{t} \Theta )}  C^{t}   \right) \;  \right) \\
 &=& {r_{t}}^{2N+1} Q_{t} \; e^{i( \lambda_{t} \Theta )}   \sum_{i=1}^{2N}    \frac{\partial}{\partial  u^{t}_{i} } 
\left[    \,  \frac{ \sqrt{ { \mathbf{g}}^{t} \, }  }{ \sqrt{\mathcal{P}} } d^{t}_{i} \; \left(  \sum_{j=1}^{2N}  d^{t}_{j} \,   \;  \frac{\partial   C^{t}  }{\partial  u^{t}_{j} }  \right) \;  \right].  
\end{eqnarray*}
From the first identity in \textbf{(a)} of Lemma \ref{magic} and the definition (\ref{Dmeans})
\[
d^{t}_{j} = \sum_{k=1}^{2N} {\lambda}_{t}  {\left( {g}^{t} \right)}^{jk} w^{t}_{k},
\]
we compute the inner sum:
\[
 \sum_{j=1}^{2N}  d^{t}_{j} \,   \;  \frac{\partial   C^{t}  }{\partial  u^{t}_{j} } 
=  \lambda_{t} \sum_{1 \leq j,k \leq 2N}      {\left( {g}^{t} \right)}^{jk} w^{t}_{k}  \frac{\partial C^{t} }{\partial  u^{t}_{j} } 
=  \lambda_{t} \left[  \mathbf{J} C^{t}  - \left(D^{t} \cdot \mathbf{J} C^{t} \right) \; D^{t} \right].    
\]
It follows that 
\begin{eqnarray*}
  { \mathbf{S} }_{2} 
 &=& \; \lambda_{t} {r_{t}}^{2N+1} Q_{t} \; e^{i( \lambda_{t} \Theta )}   \sum_{i=1}^{2N}    \frac{\partial}{\partial  u^{t}_{i} } 
\left[    \,  \frac{ \sqrt{ { \mathbf{g}}^{t} \, }  }{ \sqrt{\mathcal{P}} } d^{t}_{i} \; \left\{  \;   \mathbf{J} C^{t}  - \left(D^{t} \cdot \mathbf{J} C^{t} \right)  \; D^{t} \;   \right\} \;  \right]  \\
&=& \; \lambda_{t}{r_{t}}^{2N+1} Q_{t} \; e^{i( \lambda_{t} \Theta )}  \left[ \sum_{i=1}^{2N}    \frac{\partial}{\partial  u^{t}_{i} } 
\left(    \,  \frac{ \sqrt{ { \mathbf{g}}^{t} \, }  }{ \sqrt{\mathcal{P}} } d^{t}_{i}  \right) \; \right] \; \left\{  \;   \mathbf{J} C^{t}  - \left(D^{t} \cdot \mathbf{J} C^{t} \right)  \; D^{t} \;   \right\} \;   \\ 
&& +\, \lambda_{t} {r_{t}}^{2N+1} Q_{t} \; e^{i( \lambda_{t} \Theta )}   \sum_{i=1}^{2N}    \,  \frac{ \sqrt{ { \mathbf{g}}^{t} \, }  }{ \sqrt{\mathcal{P}} } d^{t}_{i} \;  \frac{\partial}{\partial  u^{t}_{i} } 
\left[    \;   \mathbf{J} C^{t}  - \left(D^{t} \cdot \mathbf{J} C^{t} \right)  \; D^{t} \;   \right].  \\
\end{eqnarray*}
According to the identity (\ref{sum03}) deduced in \textbf{Step B}, we notice that the sum in the first term vanishes:
\[
  \sum_{i=1}^{2N}    \frac{\partial}{\partial  u^{t}_{i} } 
\left(    \,  \frac{ \sqrt{ { \mathbf{g}}^{t} \, }  }{ \sqrt{\mathcal{P}} } d^{t}_{i}  \right)
=   {\lambda}_{t}  \sum_{1 \leq i,j \leq 2N}  \frac{\partial}{\partial  u^{t}_{i} } 
\left(    \,  \frac{ \sqrt{ { \mathbf{g}}^{t} \, }  }{ \sqrt{\mathcal{P}} }  {\left( {g}^{t} \right)}^{ij} w^{t}_{j}  \right) =0.
\]
We thus obtain 
\begin{eqnarray*}
  { \mathbf{S} }_{2} 
 &=& \;  \lambda_{t} {r_{t}}^{2N+1} Q_{t} \; e^{i( \lambda_{t} \Theta )}   \sum_{i=1}^{2N}    \,  \frac{ \sqrt{ { \mathbf{g}}^{t} \, }  }{ \sqrt{\mathcal{P}} } d^{t}_{i} \;  \frac{\partial}{\partial  u^{t}_{i} } 
\left[    \;   \mathbf{J} C^{t}  - \left(D^{t} \cdot \mathbf{J} C^{t} \right)  \, D^{t} \;   \right]  \\
 &=&   \; {\lambda_{t} r_{t}}^{2N+1} Q_{t}   \frac{ \sqrt{ { \mathbf{g}}^{t} \, }  }{ \sqrt{\mathcal{P}} }  \; e^{i( \lambda_{t} \Theta )}  \;  \mathbf{J} \; \left(  \sum_{i=1}^{2N}    \,d^{t}_{i} \;  \frac{\partial  C^{t} }{\partial  u^{t}_{i} }  \right) \\
 &&  - \lambda_{t} {r_{t}}^{2N+1} Q_{t}   \left(D^{t} \cdot \mathbf{J} C^{t} \right)   \frac{ \sqrt{ { \mathbf{g}}^{t} \, }  }{ \sqrt{\mathcal{P}} }  \; e^{i( \lambda_{t} \Theta )}   \sum_{i=1}^{2N}    \,d^{t}_{i} \;         \,  \frac{\partial D^{t}}{\partial  u^{t}_{i} }   \;     \\
&&  - \lambda_{t} {r_{t}}^{2N+1} Q_{t}   \frac{ \sqrt{ { \mathbf{g}}^{t} \, }  }{ \sqrt{\mathcal{P}} }  \; e^{i( \lambda_{t} \Theta )} \left[     \sum_{i=1}^{2N}    \,d^{t}_{i} \;  \frac{\partial}{\partial  u^{t}_{i} } 
 \;   \left(D^{t} \cdot \mathbf{J} C^{t} \right)  \,   \right]   \;  D^{t} \\
 &=&   \; {\lambda_{t}}^{2}    {r_{t}}^{2N+1} Q_{t}   \frac{ \sqrt{ { \mathbf{g}}^{t} \, }  }{ \sqrt{\mathcal{P}} }  \; e^{i( \lambda_{t} \Theta )}  \;  \mathbf{J} \; \left(   \sum_{1 \leq i,j \leq 2N}   \,       {\left( {g}^{t} \right)}^{ij} w^{t}_{j}          \;  \frac{\partial  C^{t} }{\partial  u^{t}_{i} }  \right) \\
 &&  -  {\lambda_{t}}^{2}  {r_{t}}^{2N+1} Q_{t}   \left(D^{t} \cdot \mathbf{J} C^{t} \right)   \frac{ \sqrt{ { \mathbf{g}}^{t} \, }  }{ \sqrt{\mathcal{P}} }  \; e^{i( \lambda_{t} \Theta )}   \sum_{1 \leq i,j \leq 2N}   \,       {\left( {g}^{t} \right)}^{ij} w^{t}_{j}           \;         \,  \frac{\partial D^{t}}{\partial  u^{t}_{i} }   \;     \\
&&  -  {\lambda_{t}}^{2}  {\lambda}_{t} {r_{t}}^{2N+1} Q_{t}   \frac{ \sqrt{ { \mathbf{g}}^{t} \, }  }{ \sqrt{\mathcal{P}} }  \; e^{i( \lambda_{t} \Theta )} \left[      \sum_{1 \leq i,j \leq 2N}   \,      {\left( {g}^{t} \right)}^{ij} w^{t}_{j}             \;  \frac{\partial}{\partial  u^{t}_{i} } 
 \;   \left(D^{t} \cdot \mathbf{J} C^{t} \right)  \,   \right]   \;  D^{t}. \\
\end{eqnarray*}

According to the identity \textbf{(d)} of Lemma \ref{magic}, the third sum vanishes. By using two identities in \textbf{(a)} of 
Lemma \ref{magic}, we deduce
\begin{eqnarray*}
  { \mathbf{S} }_{2} 
 &=&     {\lambda_{t}}^{2}   {r_{t}}^{2N+1} Q_{t}   \frac{ \sqrt{ { \mathbf{g}}^{t} \, }  }{ \sqrt{\mathcal{P}} }  \; e^{i( \lambda_{t} \Theta )}  \;  \mathbf{J} \; \left[ \;  \mathbf{J} C^t -  \left( D^t \cdot \mathbf{J} C^t \right) \;  D^t \;  \right] \\
 &&  -   {\lambda_{t}}^{2}  {r_{t}}^{2N+1} Q_{t}   \left(D^{t} \cdot \mathbf{J} C^{t} \right)   \frac{ \sqrt{ { \mathbf{g}}^{t} \, }  }{ \sqrt{\mathcal{P}} }  \; e^{i( \lambda_{t} \Theta )}   \left[ \; 
- \mathbf{J} D^t - \left( D^t \cdot \mathbf{J}  C^t  \right) \;  C^t  \;  \right]      \\
&=&   - {\lambda_{t}}^{2}   {r_{t}}^{2N+1} Q_{t}   \frac{ \sqrt{ { \mathbf{g}}^{t} \, }  }{ \sqrt{\mathcal{P}} }  \;   
\left[ \;  1-  {\left( D^t \cdot \mathbf{J} C^t \right)}^{2} \;  \right]  \; e^{i( \lambda_{t} \Theta )}  C^t. 
\end{eqnarray*}
\textbf{Step C3.} The identity (\ref{sum02})  
and the definition (\ref{Dmeans}) give 
\begin{eqnarray*}
 { \mathbf{S} }_{3} &=& \sum_{s=1}^{L} \sum_{i=1}^{2N}    \frac{\partial}{\partial  u^{s}_{i} }  
\left(    \sqrt{\mathbf{G}} \,  \frac{ \, -  d^{s}_{i}  \, }{ \mathcal{P} }   \frac{\partial}{\partial  \Theta }  \;   \left(  r_{t} e^{i( \lambda_{t} \Theta )}  C^{t}   \right) \;    \right)  \\
 &=&  - {\lambda}_{t} r_{t}  e^{i( \lambda_{t} \Theta )} \left[ \;   \sum_{s=1}^{L}  \sum_{i=1}^{2N}    \frac{\partial}{\partial  u^{s}_{i} }  
\left(  \sqrt{\mathbf{G}} \,  \frac{ \,   d^{s}_{i}  \, }{ \mathcal{P} }  \;  \mathbf{J}  C^{t}    \right)     \; \right]       \\
&=&  - {\lambda}_{s} r_{t}  e^{i( \lambda_{t} \Theta )} \; \mathbf{J} \; \left[ \;   \sum_{s=1}^{L}  \sum_{i=1}^{2N}     \frac{\partial}{\partial  u^{s}_{i} }  
\left(  \sqrt{\mathbf{G}} \,  \frac{ \,   d^{s}_{i}  \, }{ \mathcal{P} }  \;   C^{t}    \right)     \; \right]       \\
 &=&  - {\lambda}_{s} r_{s}  e^{i( \lambda_{t} \Theta )} \; \mathbf{J} \; \left[ \;    \sum_{s=1}^{L}    \; \left\{  \sum_{i=1}^{2N} \frac{\partial}{\partial  u^{s}_{i} }  
\left(  \sqrt{\mathbf{G}} \,  \frac{ \,  d^{s}_{i}  \, }{ \mathcal{P} }   \right)  \; \right\} \;  C^{t}      +   \sum_{i=1}^{2N}  \sum_{s=1}^{L} 
\sqrt{\mathbf{G}} \,  \frac{ \,  d^{s}_{i}  \, }{ \mathcal{P} }     \frac{\partial C^{t} }{\partial  u^{s}_{i} }     \; \right]    \\
&=&    - {\lambda}_{t} r_{t} \frac{ \,  \sqrt{\mathbf{G}}  \, }{ \mathcal{P} }   e^{i( \lambda_{t} \Theta )} \; \mathbf{J} \; \left[ \;  \sum_{i=1}^{2N}   \sum_{s=1}^{L}   d^{s}_{i}   \frac{\partial C^{t} }{\partial  u^{s}_{i} }     \; \right]     \\
&=&   -  {\lambda}_{t} r_{t} \frac{ \,  \sqrt{\mathbf{G}}  \, }{ \mathcal{P} }   e^{i( \lambda_{t} \Theta )} \; \mathbf{J} \; \left[ \;  \sum_{i=1}^{2N}     d^{t}_{i}   \frac{\partial C^{t} }{\partial  u^{t}_{i} }     \; \right]     \\
&=&  -  { {\lambda}_{t}}^{2} r_{t} \frac{ \,  \sqrt{\mathbf{G}}  \, }{ \mathcal{P} }   e^{i( \lambda_{t} \Theta )} \; \mathbf{J} \; \left[ \;     \sum_{1 \leq i,j \leq 2N}      {\left( {g}^{t} \right)}^{ij} w^{t}_{j}  \frac{\partial C^{t} }{\partial  u^{t}_{i} }     \; \right]. 
\end{eqnarray*}
The first identity in \textbf{(a)} of Lemma \ref{magic} and the factorization (\ref{FAC}) yield
\begin{eqnarray*}
 { \mathbf{S} }_{3} &=& -  { {\lambda}_{t}}^{2} r_{t} \frac{ \,  \sqrt{\mathbf{G}}  \, }{ \mathcal{P} }   e^{i( \lambda_{t} \Theta )} \; \mathbf{J} \; 
\left[ \;   \mathbf{J} C^{t}  - \left(D^{t} \cdot \mathbf{J} C^{t} \right) \; D^{t}    \; \right]  \\
&=&    { {\lambda}_{t}}^{2} r_{t} \frac{ \,  \sqrt{\mathbf{G}}  \, }{ \mathcal{P} }   e^{i( \lambda_{t} \Theta )} \; 
\left[ \;    C^{t}  + \left(D^{t} \cdot \mathbf{J} C^{t} \right) \;  \mathbf{J} D^{t}    \; \right] \\
 &=&  { {\lambda}_{t}}^{2} {r_{t}}^{2N+1} Q_{t}  \frac{ \, \sqrt{ { \mathbf{g}}^{t} \, }    \, }{ \sqrt{ \mathcal{P} }  }   e^{i( \lambda_{t} \Theta )} \;  \left[ \;    C^{t}  + \left(D^{t} \cdot \mathbf{J} C^{t} \right) \;  \mathbf{J} D^{t}    \; \right].
\end{eqnarray*}
\textbf{Step C4.} We simplify the sum  $ { \mathbf{S} }_{4}$. 
\begin{eqnarray*}
 { \mathbf{S} }_{4}
&=&     \sum_{s=1}^{L} \sum_{i=1}^{2N}    \frac{\partial}{\partial  \Theta }
\left(    \sqrt{\mathbf{G}} \,  \frac{ \, -  d^{s}_{i}  \, }{ \mathcal{P} }    \frac{\partial}{\partial  u^{s}_{i} }  \;   \left(  r_{t} e^{i( \lambda_{t} \Theta )}  C^{t}   \right) \;   \right)  \\
 &=& - \frac{ \,   \sqrt{\mathbf{G}}   \, }{ \mathcal{P} } \sum_{i=1}^{2N}    \frac{\partial}{\partial  \Theta }
\left(   \,   \sum_{s=1}^{L}     d^{s}_{i}  \frac{\partial}{\partial  u^{s}_{i} }  \;   \left(  r_{t} e^{i( \lambda_{t} \Theta )}  C^{t}   \right) \;   \right)  \\
 &=&  - \frac{ \,   \sqrt{\mathbf{G}}   \, }{ \mathcal{P} } \sum_{i=1}^{2N}    \frac{\partial}{\partial  \Theta } 
\left( \;\;\; \quad  d^{t}_{i}  \frac{\partial}{\partial  u^{t}_{i} }  \;   \left(  r_{t} e^{i( \lambda_{t} \Theta )}  C^{t}   \right) \;   \right)     \\
&=&  -   r_{t}  \frac{ \,   \sqrt{\mathbf{G}}   \, }{ \mathcal{P} } \sum_{i=1}^{2N}  d^{t}_{i}
  \frac{\partial}{\partial  \Theta } 
\left(    \frac{\partial   }{\partial  u^{t}_{i} }  \;   e^{i( \lambda_{t} \Theta )}  C^{t} \;   \right)     \\
&=&  -  {\lambda}_{t} r_{t}  \frac{ \,   \sqrt{\mathbf{G}}   \, }{ \mathcal{P} } e^{i( \lambda_{t} \Theta )}  \sum_{i=1}^{2N}  d^{t}_{i}
\left(    \frac{\partial   }{\partial  u^{t}_{i} }  \;    \mathbf{J} C^{t} \;   \right)     \\
&=&  -  {\lambda}_{t} r_{t}  \frac{ \,   \sqrt{\mathbf{G}}   \, }{ \mathcal{P} } e^{i( \lambda_{t} \Theta )}   \mathbf{J} \; \left( \; \sum_{i=1}^{2N}  d^{t}_{i}
        \frac{\partial    C^{t}  }{\partial  u^{t}_{i} }    \; \right).  
\end{eqnarray*}
From the definition $d^{t}_{i} = \sum_{j=1}^{2N} {\lambda}_{t}  {\left( {g}^{t} \right)}^{ij} w^{t}_{j}$ and the first identity in \textbf{(a)} of Lemma \ref{magic}, we deduce
\begin{eqnarray*}
 { \mathbf{S} }_{4}
&=&  - {{\lambda}_{t}}^{2} r_{t}  \frac{ \,   \sqrt{\mathbf{G}}   \, }{ \mathcal{P} } e^{i( \lambda_{t} \Theta )}  \mathbf{J} \; \left( \;   \sum_{1 \leq i,j \leq 2N}   {\left( {g}^{t} \right)}^{ij} w^{t}_{j}      \frac{\partial    C^{t}  }{\partial  u^{t}_{i} }    \; \right)     \\
&=&  - {{\lambda}_{t}}^{2} r_{t}  \frac{ \,   \sqrt{\mathbf{G}}   \, }{ \mathcal{P} } e^{i( \lambda_{t} \Theta )}  \mathbf{J} \; \left( \;   \mathbf{J} C^{t}  - \left(D^{t} \cdot \mathbf{J} C^{t} \right) \; D^{t}    \; \right)     \\
&=&   { {\lambda}_{t}}^{2} r_{t} \frac{ \,  \sqrt{\mathbf{G}}  \, }{ \mathcal{P} }   e^{i( \lambda_{t} \Theta )} \; 
\left[ \;    C^{t}  + \left(D^{t} \cdot \mathbf{J} C^{t} \right) \;  \mathbf{J} D^{t}    \; \right]. 
\end{eqnarray*}
By using the factorization (\ref{FAC}) obtained in \textbf{Step C1}: 
\[
\sqrt{\mathbf{G}} = \sqrt{ \mathcal{P}} \; {r_{t}}^{2N} \sqrt{ { \mathbf{g}}^{t} \, } Q_{t},
\]
we have
\[
 { \mathbf{S} }_{4} =  { {\lambda}_{t}}^{2} {r_{t}}^{2N+1} Q_{t}  \frac{ \, \sqrt{ { \mathbf{g}}^{t} \, }    \, }{ \sqrt{ \mathcal{P} }  }   e^{i( \lambda_{t} \Theta )} \;  \left[ \;    C^{t}  + \left(D^{t} \cdot \mathbf{J} C^{t} \right) \;  \mathbf{J} D^{t}    \; \right].
\]
\textbf{Step C5.} The term  $ { \mathbf{S} }_{5}$ can be simplified to:  
\[
 { \mathbf{S} }_{5}  =  {{\lambda}_{t}}  r_{t}  \frac{\partial}{\partial  \Theta } \;    \left(   \frac{ \,  \sqrt{\mathbf{G}}  \, }{ \mathcal{P} }   e^{i( \lambda_{t} \Theta )}  \; \mathbf{J}  C^{t}    \;  \right)   = -  {{\lambda}_{t}}^{2} r_{t} \frac{ \,  \sqrt{\mathbf{G}}  \, }{ \mathcal{P} }       e^{i( \lambda_{t} \Theta )}  \;   C^{t}.    
\]
From the factorization (\ref{FAC}) obtained in \textbf{Step C1}: 
\[
\sqrt{\mathbf{G}} = \sqrt{ \mathcal{P}} \; {r_{t}}^{2N} \sqrt{ { \mathbf{g}}^{t} \, } Q_{t},
\]
we have
\[
 { \mathbf{S} }_{5}  =   -  {{\lambda}_{t}}^{2}  {r_{t}}^{2N+1} Q_{t}  \frac{ \,  \sqrt{ { \mathbf{g}}^{t} \, }  \, }{\sqrt{ \mathcal{P}} }       e^{i( \lambda_{t} \Theta )}  \;   C^{t}.    
\]
\textbf{Step C6.} We have 
\begin{eqnarray*}
 { \mathbf{S} }_{6} &=&   \sum_{s=1}^{L}      \frac{\partial}{\partial  r^{s} }  \left( \sqrt{\mathbf{G}} 
  \frac{\partial}{\partial  r^{s} } \;   \left(  r_{t} e^{i( \lambda_{t} \Theta )}  C^{t}   \right) \;    \right).  \\
 &=& \;\;\, \quad   \frac{\partial}{\partial  r^{t} }  \left( \sqrt{\mathbf{G}} 
  \frac{\partial}{\partial  r^{t} } \;   \left(  r_{t} e^{i( \lambda_{t} \Theta )}  C^{t}   \right) \;    \right).  \\
&=&  e^{i( \lambda_{t} \Theta )}  \;  \frac{\partial  \sqrt{\mathbf{G}} }{\partial  r^{t} }   \;     C^{t}    \;     \\
&=&   \sqrt{ { \mathbf{g}}^{t} \, } Q_{t} \; e^{i( \lambda_{t} \Theta )}  \;  \left[ 2N {r_{t}}^{2N-1}  \sqrt{ \mathcal{P}}    + {r_{t}}^{2N} \frac{\partial  \sqrt{\mathcal{P}} }{\partial  r^{t} }  \right]  \;     C^{t}. 
\end{eqnarray*}
By using  (\ref{Rmeans}), (\ref{Pmeans}), and the first identity in \textbf{(b)} of Lemma \ref{magic}, we deduce  
\begin{eqnarray*}
  \frac{\partial  \sqrt{\mathcal{P}} }{\partial  r_{t} } 
&=&  \frac{1}{2   \sqrt{\mathcal{P}}  }  \frac{\partial }{\partial  r_{t} }  \left(    \mathcal{R} - \sum_{1 \leq s \leq L} \sum_{1 \leq i,j \leq 2N}   {  {\lambda}_{s} }^{2} {r_{s}}^{2}   {\left( {g}^{s} \right)}^{ij} w^{s}_{i} w^{s}_{j}   \right)  \\ 
&=&   \frac{1}{2   \sqrt{\mathcal{P}}  }    \left[  \;  2   {  {\lambda}_{t} }^{2}  r_{t}  -   2   {  {\lambda}_{t} }^{2}  r_{t}  
\sum_{1 \leq i,j \leq 2N}    {\left( {g}^{t} \right)}^{ij} w^{t}_{i} w^{t}_{j}  \;  \right]  \\ 
&=&   \frac{1}{2   \sqrt{\mathcal{P}}  }    \left[  \;  2   {  {\lambda}_{t} }^{2}  r_{t}  -   2   {  {\lambda}_{t} }^{2}  r_{t}  
 \; \left(     1 - \left(  {  D^t \cdot \mathbf{J}  C^t    }   \right)^{2}         \right) \;  \right] \\ 
&=&      \frac{  {  {\lambda}_{t} }^{2}  r_{t}   }{   \sqrt{\mathcal{P}}  }    \left(  {  D^t \cdot \mathbf{J}  C^t    }   \right)^{2}.     
\end{eqnarray*}
and meet 
\begin{eqnarray*}
 { \mathbf{S} }_{6} 
&=&  \; \; \;  2N \,  {r_{t}}^{2N-1}  Q_{t} \sqrt{ \mathcal{P}}  \sqrt{ { \mathbf{g}}^{t} \, }   \; e^{i( \lambda_{t} \Theta )}    C^{t}     \\
&&   +   {  {\lambda}_{t} }^{2}    {r_{t}}^{2N+1}  Q_{t}    \frac{ \sqrt{ { \mathbf{g}}^{t} \, }    }{   \sqrt{\mathcal{P}}  }        \left(  {  D^t \cdot \mathbf{J}  C^t    }   \right)^{2}    \; e^{i( \lambda_{t} \Theta )}  \;          C^{t}.    \\
\end{eqnarray*}
\textbf{Step C7.} Combining the results so far, we conclude
\begin{eqnarray*}
 && \; \; \; { \mathbf{S} }_{1}  +{ \mathbf{S} }_{2}  +{ \mathbf{S} }_{3}  +{ \mathbf{S} }_{4}  +{ \mathbf{S} }_{5}  +{ \mathbf{S} }_{6}  \\  
&=&    -  \; 2N \,  {r_{t}}^{2N-1}  Q_{t} \sqrt{ \mathcal{P}}  \sqrt{ { \mathbf{g}}^{t} \, }   \; e^{i( \lambda_{t} \Theta )}    C^{t}       \\  
 && - \; 2 {  {\lambda}_{t} }^{2}   {r_{t}}^{2N+1}  Q_{t}  \frac{  \sqrt{ { \mathbf{g}}^{t} \, }  }{  \sqrt{   \mathcal{P}   }}    \left(    D^t \cdot \mathbf{J}  C^t      \right)  \; e^{i( \lambda_{t} \Theta )}   \left[     \mathbf{J} D^t +  \left( D^t \cdot \mathbf{J} C^t \right) \, C^t \,    \right] \\
&&   - \;  {\lambda_{t}}^{2}   {r_{t}}^{2N+1} Q_{t}   \frac{ \sqrt{ { \mathbf{g}}^{t} \, }  }{ \sqrt{\mathcal{P}} }  \;   
\left[ \;  1-  {\left( D^t \cdot \mathbf{J} C^t \right)}^{2} \;  \right]  \; e^{i( \lambda_{t} \Theta )}  C^t \\
&& +  \; 2 \,  { {\lambda}_{t}}^{2} {r_{t}}^{2N+1} Q_{t} \frac{ \, \sqrt{ { \mathbf{g}}^{t} \, }    \, }{ \sqrt{ \mathcal{P} }  }    e^{i( \lambda_{t} \Theta )} \;  \left[ \;    C^{t}  + \left(D^{t} \cdot \mathbf{J} C^{t} \right) \;  \mathbf{J} D^{t}    \; \right]    \\  
&&   - \;  {{\lambda}_{t}}^{2}  {r_{t}}^{2N+1} Q_{t}  \frac{ \,  \sqrt{ { \mathbf{g}}^{t} \, }  \, }{\sqrt{ \mathcal{P}} }       e^{i( \lambda_{t} \Theta )}  \;   C^{t}        \\  
&& + \; 2N \,  {r_{t}}^{2N-1}  Q_{t} \sqrt{ \mathcal{P}}  \sqrt{ { \mathbf{g}}^{t} \, }   \; e^{i( \lambda_{t} \Theta )}    C^{t}     \\
&&   + \;  {  {\lambda}_{t} }^{2}    {r_{t}}^{2N+1}  Q_{t}    \frac{ \sqrt{ { \mathbf{g}}^{t} \, }    }{   \sqrt{\mathcal{P}}  }        \left(  {  D^t \cdot \mathbf{J}  C^t    }   \right)^{2}    \; e^{i( \lambda_{t} \Theta )}  \;          C^{t}    \\
&=& 0.      
\end{eqnarray*}


\end{document}